\documentclass[reqno]{amsart}

\usepackage{amsthm, amssymb, latexsym, tabls, hhline,lscape}
\usepackage{amscd}
\usepackage{amsmath}
\usepackage{tikz}
\usepackage[all]{xy}
\usepgflibrary{arrows}

\newtheorem{theorem}{Theorem}[section]
\newtheorem{proposition}[theorem]{Proposition}
\newtheorem{corollary}[theorem]{Corollary}
\newtheorem{lemma}[theorem]{Lemma}

\theoremstyle{definition}

\newtheorem{example}[theorem]{Example}

\theoremstyle{remark}
\newtheorem{remark}[theorem]{Remark}

\numberwithin{equation}{section}

\newcommand{\UIO}{\epsilon}
\newcommand{\supp}{\mathrm{supp}}

\newcommand{\larc}[1]{\hspace{-.4ex}\overset{#1}{\frown}\hspace{-.4ex}}
\newcommand{\slarc}[1]{\overset{#1}{\frown}}
\newcommand{\cS}{\mathcal{S}}
\newcommand{\ZZ}{\mathbb{Z}}
\newcommand{\FF}{\mathbb{F}}
\newcommand{\Id}{\mathrm{Id}}
\newcommand{\scscs}{\scriptscriptstyle}
\newcommand{\fkS}{\mathfrak{S}}

\newcommand{\Res}{\mathrm{Res}}

\newcommand{\cX}{\mathcal{X}}
\newcommand{\cK}{\mathcal{K}}
\newcommand{\QQ}{\mathbb{Q}}
\newcommand{\CC}{\mathbb{C}}
\newcommand{\SC}[3]{{{\bf scf}(U^{#2})}_{#1}#3}
\newcommand{\Inf}{\mathrm{Inf}}
\newcommand{\nst}{\mathrm{nst}}

\makeindex

\title[Antipode and Primitive elements]
{Antipode and Primitive elements  in the Hopf Monoid of Super Characters}

\author{D. Baker-Jarvis}\address[Duff Baker-Jarvis]
{University of Colorado \textbf{Boulder}}
\email{Duff.Baker-jarvis@Colorado.EDU}

\author{N. Bergeron}\address[Nantel Bergeron]
{Department of Mathematics and Statistics\\ York  University\\ To\-ron\-to, Ontario M3J 1P3\\ CANADA}
\email{bergeron@mathstat.yorku.ca}
\urladdr{http://www.math.yorku.ca/bergeron}

 \author{N. Thiem}\address[Nathaniel Thiem]
 {University of Colorado \textbf{Boulder}}
 \email{Nathaniel.Thiem@Colorado.EDU}

\date{\today}

\thanks{N. Bergeron is supported in part by CRC and NSERC}
\thanks{D. Baker-Jarvis and N. Thiem are supported in part by NSF FRG DMS-0854893}


\subjclass[2000]{}

\begin{document}

\begin{abstract}
From a recent paper, we recall the Hopf monoid structure on the supercharacters of the unipotent uppertriangular groups over a finite field.
We give cancelation free formula for the antipode applied to the bases of class functions and power sum functions, giving
 new cancelation free formulae for the standard Hopf algebra of supercharacters and symmetric functions in noncommuting variables.
We also give partial results for the antipode on the super character basis, and
 explicitly describe  the primitives of this Hopf monoid.
\end{abstract}

\maketitle

\section{Introduction}\setcounter{equation}{0}

The Hopf algebra of symmetric functions in non-commuting variables (known as $\mathrm{NCSym}$, $\mathbf{WSym}$, or $\Pi$) has  a representation theoretic interpretation via the supecharacters of the finite groups of unipotent uppertriangular matrices \cite{BDTT10}, where the product comes from an inflation functor on representations and the coproduct comes from the restriction functor on representations.    

As seen in \cite{ABT}, the Hopf algebra $\Pi$  can be viewed as a functorial image of the Hopf monoid ${\bf scf}(U)$ of superclass functions on unipotent uppertriangular groups over a finite field of cardinality 2.  
The interest of such a construction is both conceptual and computational.
In this paper we concentrate on computational aspects. We  derive cancelation free formula for the antipode in some bases 
and describe explicitly the primitives of the Hopf monoid. 
From the functorial properties, this gives us a  new cancelation free formula for the antipode in $\Pi$
and a better understanding of its structure.

 The notion of a supercharacter theory of a finite group was introduced by Diaconis--Isaacs \cite{DI08} to generalize an approach used by Andr\'e (eg. \cite{An95}) and Yan \cite{Ya01} to study the characters of the finite groups of unipotent upper-triangular matrices.   The basic idea is to coarsen the usual character theory of a group by replacing irreducible characters with linear combinations of irreducible characters that are constant on a set of clumped conjugacy classes, called superclasses.   We focus here on the supercharacters of the finite groups of upper-triangular matrices $U_n(q)$ over a field with $q$ elements.  
 
 Most finite groups have more than one possible supercharacter theory, so there is not really a canonical choice. However, Keller  \cite{K} shows  that for every group there is a unique finest  supercharacter theory with integer values. Our heuristic is that this finest integral supercharacter theory should play a special role in mathematics. 
For the groups $U_n(q)$, while the usual supercharacter theory used in \cite{BDTT10, ABT,  An95, DI08,Ya01} is not integer valued, there is a natural coarsening (used in \cite{BT}) of this supercharacter theory that is integer valued.  We do not know whether this is the finest integral supercharacter theory, but it seems to be even more (combinatorially)  natural  than the usual one. For example, its supercharacter table has a nice decomposition as a lower triangular matrix times an upper triangular matrix \cite{BT}, and for each $q$, the corresponding Hopf algebra is always isomorphic to the Hopf algebra of symmetric functions in non-commuting variables (\cite{BDTT10} required $q=2$ for that supercharacter theory). This allows us to deform bases as $q$ varies. This paper describes the Hopf monoid ${\bf scf}(U)$ using this integral supercharacter theory instead of the usual one. For this reason some of the statements will be slightly different than in \cite{ABT}, but the proofs are  the same. 

  For $U_n(q)$, each subset $K\subseteq \{1,2,\ldots, n\}$ gives a subgroup $U_K(q)\subseteq U_n(q)$ isomorphic to $U_{|K|}(q)$, obtained by taking only the rows and columns of $K$.  It turns out that restriction to the subgroup $U_{|K|}(q)$ depends on the subset $K$.  Our construction of the Hopf monoid takes these distinctions into account by simultaneously considering all groups $U_K(q)$ with all possible linear orders on the sets $K$.  The usual Hopf algebra is obtained by following the isomorphism $U_K(q)\cong U_{|K|}(q)$.  However, a postponement of this identification  allows us to work with a finer structure.  While in general a supercharacter theory of a group $G$ is not compatible with a supercharacter theory of a  subgroup $H\subseteq G$, the supercharacter theories of $U_n(q)$ as $n$ varies behave nicely in the sense that the restriction functor sends functions constant on superclasses to functions constant on superclasses of the subgroup.

Bases of ${\bf scf}(U)$ are indexed by a generalization of set partitions that also encode the linear order of their underlying set. 
In analogy with the ring of symmetric functions, we consider three distinguished bases for ${\bf scf}(U)$: the $\kappa$-basis corresponds to superclass characteristic functions (the monomial symmetric functions), the $P$-basis (called the $\lambda$-basis in \cite{ABT}) corresponds to power-sum symmetric functions (with some variations), and the $\chi$-basis corresponds to the basis of supercharacters.    While it might be tempting to think of the $\chi$-basis as an analogue of Schur functions, this analogy seems to be misleading.  

The main thrust of the paper is to develop antipode formulas for these distinguished bases 
 and to describe the primitives of ${\bf scf}(U)$.  
Sections~\ref{sec:Pbasis} and \ref{sec:Mbasis}  find pleasing formulas for the $\kappa$ and $P$ bases, respectively, and Section~\ref{sec:Xbasis} finds a formula for a family of basis elements in the $\chi$-basis.  This last problem seems to be much more involved, and even the basic cases we consider lead to some surprising combinatorics on ribbon-shaped tableaux with new identities.
In Section~\ref{sec:primitive} we describe the primitives of ${\bf scf}(U)$. 
In some sections we show how our formulas implies new stronger results for the standard Hopf algebra  $\Pi$.


\section{Preliminaries}

We will use the theory on Hopf monoid on species developed by Aguiar and Mahajan in~\cite{AM}.
There are shorter introductions to the subject,  such as~\cite{AM12} or~\cite[Section 1]{ABT}. Most of our notation follows~\cite{ABT,AM,AM12},
but for labelled set partitions we prefer a notation closer to~\cite{BDTT10}.
\subsection{Set partitions with linear orders} 
The Hopf monoid ${\bf scf}(U)$ is constructed from set partitions with linear orders. 

Given a set $K$, let $L[K]$ be the set of linear orders on $K$.  Let $[n]=\{1,2,\ldots,n\}$.    Given $\phi\in L[K]$, a \emph{set partition $\lambda$ of $K$ with respect to $\phi$} is a set $\lambda$ of pairs $i\larc{}j$ with $i,j\in K$ such that 
\begin{enumerate}
\item[(a)] $i\larc{}j\in \lambda$ implies $i\prec_\phi j$,
\item[(b)] if $i\larc{}l,j\larc{}k\in \lambda$, then $i=j$ if and only if $k=l$.
\end{enumerate}
For $\phi\in L[K]$, let
  $$\cS^\phi=\left\{\begin{array}{c}\text{Set partitions $\lambda$ of}\\ \text{$K$ with respect to $\phi$}\end{array}\right\}.$$
Note that the linear order $\phi$ already contains the information of the underlying set $K$, so we will typically suppress $K$ in the notation.
Such a set partition $\lambda$ gives a set partition of $K$ (in the usual sense) by the rule that $i,j\in K$ are in the same part if $i\larc{}j\in \lambda$.

\begin{remark}
Note that since we have defined $\lambda$ as a set of pairs, it is quite possible that $\lambda\in \cS^\phi\cap \cS^\tau$ for two different linear orders $\phi$ and $\tau$ (possibly even on different sets).  We make use of this ambiguity, but will also add the linear order to the notation if we need it fixed.  For example, we might write $(\phi,\lambda)$ instead of $\lambda$ if we want to specify that $\lambda\in \cS^\phi$.
\end{remark}

It is convenient to draw the set partition $\lambda$ above the total 
order $\phi$ by drawing an arc between elements $i,j$ when $i\larc{}j\in \lambda$. For example, the set partition $\{1\larc{}9,9\larc{}2,3\larc{} 8,6\larc{}4\}$  of $[9]$ with respect to the order $614925378$ corresponds to the set partition $\{\{1,2,9\},\{3,8\},\{4,6\}\}$ of $[9]$,
and is depicted as
$$
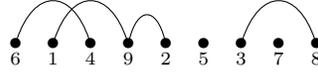
\begin{tikzpicture}[baseline=.2cm]
	\foreach \x in {1,2,3,4,5,6,7,8,9} 
		\node (\x) at (\x/2,0) [inner sep=-1pt] {$\bullet$};
	\node at (1/2,-.2) {$\scriptstyle 6$};
	\node at (2/2,-.2) {$\scriptstyle 1$};
	\node at (3/2,-.2) {$\scriptstyle 4$};
	\node at (4/2,-.2) {$\scriptstyle 9$};
	\node at (5/2,-.2) {$\scriptstyle 2$};
	\node at (6/2,-.2) {$\scriptstyle 5$};
	\node at (7/2,-.2) {$\scriptstyle 3$};
	\node at (8/2,-.2) {$\scriptstyle 7$};
	\node at (9/2,-.2) {$\scriptstyle 8$};
	\draw (2) .. controls (2.5/2,.75) and (3.5/2,.75) ..  node [above=-2pt] { } (4); 
	\draw (4) .. controls (4.25/2,.5) and (4.75/2,.5) ..  node [above=-2pt] { } (5); 
	\draw (7) .. controls (7.5/2,.75) and (8.5/2,.75) ..  node [above=-2pt] { } (9); 
	\draw (1) .. controls (1.5/2,.75) and (2.5/2,.75) ..  node [above=-2pt] { } (3); 
\end{tikzpicture} 
$$

We have two natural statistics on set partitions, the \emph{dimension} $\dim(\lambda)$ and the \emph{nesting statistic} $\nst_\mu^\lambda$ on a pair.  For $\lambda\in \cS^\phi$, let 
$$\dim(\lambda)=\sum_{i\slarc{}k\in\lambda}\#\{i\prec_\phi j\preceq_\phi k\},$$
and for $\lambda,\mu\in \cS^\phi$, let
$$\nst^\lambda_\mu=\#\{i\prec_\phi j\prec_\phi k\prec_\phi l\mid i\larc{}l\in \lambda, j\larc{}k\in \mu\}.$$

\subsection{Supercharacters of $U^\phi$} Most of the material in this section is adapted from \cite{BDTT10,DI08}.
A \emph{supercharacter theory} of a group $G$ is a pair $(\cK,\cX)$ where $\cK$ is set partition of $G$ such that the parts are unions of conjugacy classes, and $\cX$ is a set partition of the irreducible characters of $G$, such that 
\begin{itemize}
\item $|\cK|=|\cX|$,
\item For each $X\in \cX$, the character $\sum_{\chi\in X}\chi(1)\chi$ is constant on the parts of $\cK$.
\item The identity of $G$ is in its own part of $\cK$ and the trivial character of $G$ is in its own part of $\cX$.
\end{itemize} 
We refer to the parts of $\cK$ as \emph{superclasses}.  For each $X\in \cX$, fix a choice of $c_X\in \QQ_{\geq 0}$ such  that 
\begin{equation}\label{CoefficientChoice}
\chi^X=c_X\sum_{\chi\in X} \chi(1)\chi
\end{equation}
is still a character.   Then the $\{\chi^X\mid X\in \cX\}$ are the \emph{supercharacters} of $(\cK,\cX)$.

Let $\UIO_n\in L[n]$ denote the usual order on $[n]$ given by
\begin{equation}\label{UsualLinearOrder}
1\prec_{\UIO_n} 2\prec_{\UIO_n} \cdots \prec_{\UIO_n} n.
\end{equation}  
While many possible supercharacter theories are possible, we will focus on a specific supercharacter theory for the finite groups of unipotent upper-triangular matrices
$$U^{\UIO_n}:=U_n(q)=\{u\in \mathrm{GL}_n(\FF_q)\mid u_{ii}=1, 1\leq i\leq n, u_{ji}=0, 1\leq i<j\leq n\},$$
where $\FF_q$ is the finite field with $q$ elements.  Let 
$$T_n(q)=\{t\in \mathrm{GL}_n(\FF_q)\mid t_{ij}=0, i\neq j\}$$
be the subgroup of diagonal matrices.  Define an action of $U^{\UIO_n}\times U^{\UIO_n}\times T_n(q)$ on the set of nilpotent uppertriangular matrices $U^{\UIO_n}-\mathrm{Id}$ by 
$$(a,b,t)(u-\Id)= t(a(u-\Id)b^{-1})t^{-1}, \qquad \text{for $a,b\in U^{\UIO_n}, t\in T_n(q)$}.$$
Two elements $u,v\in U^{\UIO_n}$ are in the same superclass, if $u-\mathrm{Id}$ and $v-\mathrm{Id}$ are in the same orbit in $U^{\UIO_n}-\mathrm{Id}$.  Each of these orbits has a unique matrix that has at most one 1 in every row and column and zeroes elsewhere, so
$$\cK\longleftrightarrow \left\{\begin{array}{c} u\in (U^{\UIO_n}-\mathrm{Id}),\text{ with at}\\  \text{most one 1 in every row }\\ \text{and column and 0's elsewhere}\end{array}\right\}\longleftrightarrow \cS^{\UIO_n},$$
where the second bijection is obtained by letting the coordinates of the nonzero entries give the arcs of the set partition.

The supercharacters are also indexed by $\cS^{\UIO_n}$, and are easiest to describe via their character formula.   For $\lambda,\mu\in \cS^{\UIO_n}$ and $u-\Id$ with superclass type $\mu$,
\begin{equation}\label{SupercharacterFormula}
\chi^\lambda(u)=\left\{\begin{array}{ll} \frac{(-1)^{|\lambda\cap \mu|} q^{\dim(\lambda)-|\lambda|}(q-1)^{|\lambda-\mu|}}{q^{\nst_\mu^\lambda}}, & \begin{array}{@{}l@{}} \text{if $i\prec_{\UIO_n}j\prec_{\UIO_n}k$, $i\larc{}k\in \lambda$}\\ \text{implies $i\larc{}j,j\larc{}k\notin \mu$,}\end{array}\\ 0, & \text{otherwise.}\end{array}\right.
\end{equation}

In this supercharacter theory, our choice of constants in (\ref{CoefficientChoice}) is
$$c_\lambda= \frac{q^{|C(\lambda)|}(q-1)^{|\lambda|}}{\chi^\lambda(1)},\quad \text{where}\quad C(\lambda)=\{ i\prec_{\UIO_n}j\prec_{\UIO_n}k\prec_{\UIO_n}l \mid i\larc{}k,j\larc{}l\in \lambda\}.$$

For each finite set $K$ and for $\phi\in L[K]$ we obtain a group
$$U^\phi=\{u\in \mathrm{GL}_K^\phi(\FF_q)\mid u_{ii}=1, i\in K, u_{ji}=0, \text{ if $i\preceq_\phi j$}\},$$
where $\mathrm{GL}_K^\phi(\FF_q)$ is the group of invertible $|K|\times|K|$ matrices whose rows and columns are indexed by $K$ in the order prescribed by $\phi$, and whose entries are in $\FF_q$.  Then $U^\phi\cong U^{\UIO_{|K|}}$ and it has a corresponding supercharacter theory indexed by $\cS^\phi$.  We will add the underlying order to the label on the supercharacter, so if $\lambda\in \cS^\phi$, then the corresponding supercharacter is $\chi^{(\phi,\lambda)}$.  Let
$$\SC{}{\phi}{}=\CC\text{-span}\{\chi^{(\phi,\lambda)}\mid  \lambda\in \cS^\phi\}$$
be the vector space spanned by the supercharacters of $U^\phi$.

A {\sl set composition}  $(J_1,J_2,\ldots, J_\ell)$ of a set $K$ is a sequence of pairwise disjoint subsets $J_1,J_2,\ldots, J_\ell\subseteq K$ such that $K=J_1\cup J_2\cup\cdots\cup J_\ell$.
In this case, we write $(J_1,J_2,\ldots, J_\ell)\models K$.  If $(J_1, J_2,\ldots, J_\ell)\models K$ is a set composition of $K$, then there is a natural injective homomorphism
$$\iota\colon U^{\phi|_{J_1}}\times \cdots \times U^{\phi|_{J_\ell}} \longrightarrow U^\phi,$$
given by
$$\iota(u^{(1)},\ldots, u^{(\ell)})_{jk} = \left\{\begin{array}{ll} u_{jk}^{(i)} & \text{if $j,k\in J^{(i)}, 1\leq i\leq \ell$}, \\ 0 & \text{otherwise,}\end{array}\right.\quad \text{for $u^{(i)}\in U^{\phi|_{J_i}}$, $1\leq i\leq \ell$}.$$
This injective function gives a restriction map
$$\begin{array}{rccc}\Res^{U^\phi}_{U^{\phi|_{J_1}}\times \cdots \times U^{\phi|_{J_\ell}}}\colon & \SC{}{\phi}{} & \longrightarrow  & \SC{}{\phi|_{J_1}}{}\otimes \cdots \otimes \SC{}{\phi|_{J_\ell}}{}\\
&\chi & \mapsto & \Res_{\iota(U^{\phi|_{J_1}}\times \cdots \times U^{\phi|_{J_\ell}})}^{U^\phi}(\chi).\end{array}$$
If $J\subseteq K$, then we may view $U^{\phi|_J}\subseteq U^\phi$ by identifying $J$ with the set composition of $K$ with all block sizes 1 except for one block $J$. 

The following proposition is an adaptation of results proved in \cite{Th10} to our slightly coarser supercharacter theory with the exception of (4), which follows from a simple counting argument applied to the supercharacter formula (\ref{SupercharacterFormula}). 

\begin{proposition}\label{RestrictionProposition}
Let $(J_1,J_2,\ldots, J_\ell)$ be a set composition of $K$, and let $\phi\in L[K]$. Then 
\begin{enumerate}
\item For $\lambda\in \cS^\phi$,
$$\Res^{U^\phi}_{U^{\phi|_{J_1}}\times \cdots \times U^{\phi|_{J_\ell}}}(\chi^{(\phi,\lambda)})=\frac{1}{\chi^\lambda(1)^{\ell-1}}\bigotimes_{i=1}^\ell \Res^{U^\phi}_{U^{\phi|_{J_i}}}(\chi^{(\phi,\lambda)}).$$
\item For $J\subseteq K$, $\lambda\in \cS^\phi$,
$$\Res^{U^\phi}_{U^{\phi|_J}}(\chi^{(\phi,\lambda)})=\bigodot_{i\slarc{}j\in \lambda} \Res^{U^\phi}_{U^{\phi|_J}}(\chi^{(\phi,\{i\slarc{}j\})}).$$
\item For $i\prec_\phi l$, $J\subseteq K$, and $t=q-1$, 
$$\Res^{U^\phi}_{U^{\phi|_J}}(\chi^{(\phi,i\slarc{}l)})=\left\{\begin{array}{@{}ll} q^{a} \chi^{(\phi|_J,i\slarc{}l)} & \text{if $i,l\in J$,}\\
\displaystyle q^{a} t\big(\chi^{(\phi|_J,\emptyset)} +\sum_{j\in J\atop i\prec_\phi j \prec_\phi l} \chi^{(\phi|_J,j\slarc{}l)}\big) & \text{if $i\notin J$, $l\in J$},\\
\displaystyle q^{a}t \big(\chi^{(\phi|_J,\emptyset)} +\sum_{j\in J\atop i\prec_\phi j \prec_\phi l} \chi^{(\phi|_J,i\slarc{}j)}\big) & \text{if $i\in J$, $l\notin J$},\\
\displaystyle q^{a}t \big((bt+1)\chi^{(\phi|_J,\emptyset)} +t\hspace{-.4cm}\sum_{j,k\in J\atop i\prec_\phi j\prec_\phi k \prec_\phi l} \hspace{-.4cm}\chi^{(\phi|_J,j\slarc{}k)}\big) & \text{if $i,l\notin J$},\end{array}\right.$$
where $a=\#\{j\notin J\mid i\prec_\phi\! j\prec_\phi l\}$ and $b=\#\{j\in J\mid i\prec_\phi j\prec_\phi l\}$.
\item For $\lambda\in \cS^\phi$,
$$\chi^{(\phi,\lambda)}(1)^{\ell-1}=(q-1)^{|\lambda|}\prod_{m=1}^{\ell} \prod_{i\slarc{}l\in \lambda} q^{\#\{j\notin J\mid i\prec_\phi\! j\prec_\phi l\}}.$$
\end{enumerate}
\end{proposition}

If $(J_1, J_2,\ldots, J_\ell)\models K$ is a set composition and $\phi\in L[K]$ such that $\phi=\phi_1\phi_2\cdots \phi_\ell$ with $\phi_i\in L[J_i]$, then there exists a surjective homomorphism
$$\rho\colon U^\phi \longrightarrow U^{\phi_1}\times \cdots \times U^{\phi_\ell},$$
such that 
$$\rho\circ \iota(u)=u,\qquad \text{for $u\in U^{\phi_1}\times \cdots \times U^{\phi_\ell}$.}$$
We therefore obtain an inflation map
$$\Inf^{U^\phi}_{U^{\phi_1}\times \cdots \times U^{\phi_\ell}}\colon   \SC{}{\phi_1}{}\otimes \cdots \otimes \SC{}{\phi_\ell}{} \longrightarrow   \SC{}{\phi}{}$$
given by
$$\Inf^{U^\phi}_{U^{\phi_1}\times \cdots \times U^{\phi_\ell}}(\chi_1\otimes \cdots \otimes \chi_\ell)(u)= (\chi_1\otimes \cdots \otimes \chi_\ell)(\rho(u)), \quad \text{for $u\in U^\phi$.}$$
Fortunately, inflation has a nice description on supercharacters.
\begin{proposition} Let $\phi\in L[I]$ and $\tau\in L[J]$ with $I\cap J=\emptyset$.
For $\lambda\in \cS^\phi$ and $\mu\in \cS^\tau$,
$$\Inf_{U^{\phi}\times U^\tau}^{U^{\phi\tau}}(\chi^{(\phi,\lambda)}\otimes \chi^{(\tau,\mu)})=\chi^{(\phi\tau,\lambda\cup \mu)}.$$
\end{proposition}

Each space $\SC{}{\phi}{}$ has a second natural basis spanned by the superclass characteristic functions 
$$\kappa_{(\phi,\mu)}(u)=\left\{\begin{array}{ll} 1 & \begin{array}{@{}l@{}}\text{if $u$ is in the superclass}\\ \text{indexed by $\mu$,}\end{array}\\ 0 & \text{otherwise,}\end{array}\right.\qquad \text{for $\mu\in \cS^\phi, u\in U^\phi$.}$$
This basis also has a nice description with respect to the maps $\Res$ and $\Inf$.  For $J\subseteq K$, $\phi\in L[K]$ and $\mu\in \cS^\phi$, let
\begin{equation}\label{RestrictToSubset}
\mu_J=\{i\larc{}j\in \mu\mid i,j\in J\}.
\end{equation}
Note that if either $i$ or $j$ is not in $J$ then $i\larc{}j\notin\mu_J$.  

\begin{proposition} \label{FunctorsSuperclasses} Let  $I$ and $J$ be sets with $I\cap J=\emptyset$.  Then
\begin{enumerate}
\item For $\phi\in L[I\cup J]$ and $\mu\in \cS_{}^\phi$,
$$\Res^{U^\phi}_{U^{\phi|_I}\times U^{\phi|_J}}(\kappa_{(\phi,\mu)})
=\left\{\begin{array}{ll}
\kappa_{(\phi|_I,\mu_I)}\otimes \kappa_{(\phi|_J,\mu_J)} & \text{if $\mu=\mu_I\cup \mu_J$,}\\ 
0 & \text{otherwise.}
\end{array}\right.$$
\item For $\phi\in L[I]$, $\tau\in L[J]$, $\mu\in \cS^\phi$ and $\nu\in \cS^\tau$,
$$
\Inf^{U^{\phi\tau}}_{U^{\phi}\times U^{\tau}}(\kappa_{(\phi,\mu)}\otimes \kappa_{(\tau,\nu)})=\sum_{\lambda\in \cS^{\phi\tau}\atop \lambda_I=\mu, \lambda_J=\nu} \kappa_{(\phi\tau,\lambda)}.
$$
\end{enumerate}
\end{proposition}

\section{The Hopf monoid of supercharacters} \label{sec:SC}

Here we adapt the results of~\cite{ABT} to our coarser supercharacter theory with the notation introduced above.
For a fixed $q$, the Hopf monoid ${{\bf scf}(U)}$ is defined as follow. For any finite set $K$, let
$$\SC{}{}{[K]}=\bigoplus_{\phi\in L[K]} \SC{}{\phi}{}\,.$$
For $K=I\sqcup J$, the Hopf monoid $\SC{}{}{}$ has product $m_{I,J}$ (typically denoted by ``$\cdot$'') given by
\begin{equation}\label{SCProduct}
\begin{array}{c@{\ }c@{\ }c}\cdot\colon\SC{}{\phi}{}\otimes \SC{}{\tau}{}& \to & \SC{}{\phi \tau}{}\\
\chi\otimes \psi & \mapsto & \Inf_{U^\phi\times U^\tau}^{U^{\phi\tau}}(\chi\otimes \psi)\end{array} \quad\text{for $\phi\in L[I]$ and $\tau\in L[J]$,}
\end{equation}
and coproduct $\Delta_{I,J}$ given by its restriction to the various subspaces $\SC{}{\phi}{}$
 \begin{equation}\label{SCCoproduct}
\begin{array}{r@{}c@{\ }c@{\ }c} \Res_{\SC{}{\phi}{}}^{\SC{}{}{[K]}} (\Delta_{I,J})\colon & \SC{}{\phi}{}) & \to & \SC{}{\phi|_I}{}\otimes \SC{}{\phi|_J}{}\\
& \chi & \mapsto & \Res_{U^{\phi|_I}\times U^{\phi|_J}}^{U^\phi}(\chi),\end{array} \quad \text{for $\phi\in L[K]$.}
\end{equation}

We will use Takeuchi's formula several times in this paper.  Given a finite set $K$ and a set composition $J=(J_1, J_2, \cdots, J_\ell)\models K$ we define
 $$m_{J}=m_{K\setminus J_{\ell},\, J_{\ell}} \circ \cdots \circ \big(m_{J_{1},J_{2}} \otimes \hbox{Id}_{J_3}\otimes\cdots\otimes \hbox{Id}_{J_\ell}\big).$$
 By convention, $m_{J}=\hbox{Id}_K$ if $\ell=1$. Similarly we let $\Delta_{(K)}=\hbox{Id}_K$ and in general
  $$\Delta_{J}=\big(\Delta_{J_{1},J_{2}} \otimes \hbox{Id}_{J_3}\otimes\cdots\otimes \hbox{Id}_{J_\ell}\big) \circ \cdots \circ \Delta_{K\setminus J_{\ell},\, J_{\ell}}.$$
 Then Takeuchi's formula says that for $\psi\in \SC{}{}{[K]}$,
 \begin{equation}\label{TakeuchisFormula}
 S(\psi)=\sum_{J=(J_1,J_2,\ldots,J_\ell)\models K} (-1)^{\ell} m_J\circ \Delta_J(\psi).
 \end{equation}

\begin{remark} \label{rem:functor}
Under the action of $\fkS_n$, the coinvariant class of any element $\chi^{(\phi,\lambda)}\in\SC{}{\phi}{}$ with $|\supp(\phi)|=n$ can be identified with the element 
$\chi^{(\UIO_n,\phi^{-1}\circ\lambda)}\in\SC{}{\UIO_n}{}$, where $\phi^{-1}\circ\lambda$ is the set of arcs one gets by relabeling the arcs according to the permutation $\phi^{-1}$. Visually, this corresponds to simply replacing $\phi$ by $\UIO_n$ and keeping the arcs the same:
$$ 
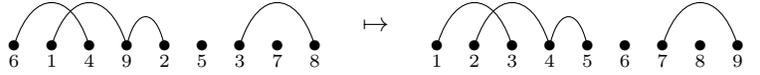
\begin{tikzpicture}[baseline=.2cm]
	\foreach \x in {1,2,3,4,5,6,7,8,9} 
		\node (\x) at (\x/2,0) [inner sep=-1pt] {$\bullet$};
	\node at (1/2,-.2) {$\scriptstyle 6$};
	\node at (2/2,-.2) {$\scriptstyle 1$};
	\node at (3/2,-.2) {$\scriptstyle 4$};
	\node at (4/2,-.2) {$\scriptstyle 9$};
	\node at (5/2,-.2) {$\scriptstyle 2$};
	\node at (6/2,-.2) {$\scriptstyle 5$};
	\node at (7/2,-.2) {$\scriptstyle 3$};
	\node at (8/2,-.2) {$\scriptstyle 7$};
	\node at (9/2,-.2) {$\scriptstyle 8$};
	\draw (2) .. controls (2.5/2,.75) and (3.5/2,.75) ..  node [above=-2pt] { } (4); 
	\draw (4) .. controls (4.25/2,.5) and (4.75/2,.5) ..  node [above=-2pt] { } (5); 
	\draw (7) .. controls (7.5/2,.75) and (8.5/2,.75) ..  node [above=-2pt] { } (9); 
	\draw (1) .. controls (1.5/2,.75) and (2.5/2,.75) ..  node [above=-2pt] { } (3); 
\end{tikzpicture} 
\quad\mapsto\quad
\begin{tikzpicture}[baseline=.2cm]
	\foreach \x in {1,2,3,4,5,6,7,8,9} 
		\node (\x) at (\x/2,0) [inner sep=-1pt] {$\bullet$};
	\node at (1/2,-.2) {$\scriptstyle 1$};
	\node at (2/2,-.2) {$\scriptstyle 2$};
	\node at (3/2,-.2) {$\scriptstyle 3$};
	\node at (4/2,-.2) {$\scriptstyle 4$};
	\node at (5/2,-.2) {$\scriptstyle 5$};
	\node at (6/2,-.2) {$\scriptstyle 6$};
	\node at (7/2,-.2) {$\scriptstyle 7$};
	\node at (8/2,-.2) {$\scriptstyle 8$};
	\node at (9/2,-.2) {$\scriptstyle 9$};
	\draw (2) .. controls (2.5/2,.75) and (3.5/2,.75) ..  node [above=-2pt] { } (4); 
	\draw (4) .. controls (4.25/2,.5) and (4.75/2,.5) ..  node [above=-2pt] { } (5); 
	\draw (7) .. controls (7.5/2,.75) and (8.5/2,.75) ..  node [above=-2pt] { } (9); 
	\draw (1) .. controls (1.5/2,.75) and (2.5/2,.75) ..  node [above=-2pt] { } (3); 
\end{tikzpicture} \ .
$$
We have that $\overline{\mathcal K}({{\bf scf}(U)})$ is isomorphic to $\Pi$, as described in~\cite{ABT}.  In $\Pi$, any set of arcs $\lambda$  is assumed to be in $\cS^{\UIO_n}$.
\end{remark}

\section{Antipode on the P-bases}\label{sec:Pbasis}
Lauve and Mastnak  give a formula for the antipode of the $P$ basis in the (standard) graded Hopf algebra $\Pi$ \cite{LM}. The formula is very nice but may still contain some 
cancelation in its expansion. We give here a formula for a generalization of $P$ bases at the Hopf monoid level, show that it is multiplicity and cancelation free, and that the functorial image to the Hopf algebra $\Pi$ is cancelation free as well. 

\subsection{Atomic set partitions}

Let $\phi\in L[K]$. A set partition $\lambda\in \cS^\phi$ is \emph{atomic} if there is no set composition $(A,B)\models K$ such that 
$$\phi=\phi|_{A}\phi|_{B}\qquad \text{and}\qquad \lambda=\lambda_{A}\cup \lambda_{B}.$$
Roughly, speaking this means one cannot obtain an atomic set partition by concatenating two other set partitions.

We will want a slightly more general notion of atomic that captures how atomic a set partition can be with respect to a different order.   Given two $\phi,\tau\in L[K]$ there is a unique factorization $\phi=\phi_1\phi_2 \cdots\phi_k$ into rising subsequences with respect to the order $\tau$, where we start a new subsequence at each descent in $\phi$ with respect to $\tau$.  For example, if  $\phi=614925378$ and $\tau=123456789$, then the decomposition into rising subsequences is given by $\phi=6\cdot149\cdot25\cdot378$.

For $\phi,\tau\in L[K]$, a set partition $\lambda\in \cS^\tau$ is \emph{$\phi$-atomic} if the set composition $(J_1,J_2,\ldots, J_\ell)$ of $K$ coming from the factorization 
$\tau=\tau|_{J_1}\cdots \tau|_{J_\ell}$ into maximal rising subsequences with respect to $\phi$ satisfies
\begin{enumerate}
\item[(a)] $\lambda=\lambda_{J_1}\cup \lambda_{J_2}\cup \cdots \cup \lambda_{J_\ell}$,
\item[(b)] $\lambda_{J_k}$ is atomic for all $1\leq k\leq \ell$. 
\end{enumerate}
If $\lambda$ is not $\phi$-atomic, then we have two possibilities:
\begin{itemize}
\item if $\lambda\notin \cS^\phi$, then $\lambda$ is \emph{$\phi$-incompatible},
\item  if $\lambda\in \cS^\phi$, then $\lambda$ is \emph{$\phi$-decomposable}.
\end{itemize}
Note that $\phi$-atomic matches our definition for atomic when $\tau=\phi$.

For example,
$$
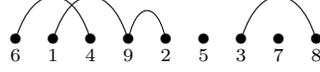
\begin{tikzpicture}[baseline=0cm]
	\foreach \x in {1,2,3,4,5,6,7,8,9} 
		\node (\x) at (\x/2,0) [inner sep=-1pt] {$\bullet$};
	\node at (1/2,-.2) {$\scriptstyle 6$};
	\node at (2/2,-.2) {$\scriptstyle 1$};
	\node at (3/2,-.2) {$\scriptstyle 4$};
	\node at (4/2,-.2) {$\scriptstyle 9$};
	\node at (5/2,-.2) {$\scriptstyle 2$};
	\node at (6/2,-.2) {$\scriptstyle 5$};
	\node at (7/2,-.2) {$\scriptstyle 3$};
	\node at (8/2,-.2) {$\scriptstyle 7$};
	\node at (9/2,-.2) {$\scriptstyle 8$};
	\draw (2) .. controls (2.5/2,.75) and (3.5/2,.75) ..  node [above=-2pt] { } (4); 
	\draw (4) .. controls (4.25/2,.5) and (4.75/2,.5) ..  node [above=-2pt] { } (5); 
	\draw (7) .. controls (7.5/2,.75) and (8.5/2,.75) ..  node [above=-2pt] { } (9); 
	\draw (1) .. controls (1.5/2,.75) and (2.5/2,.75) ..  node [above=-2pt] { } (3); 
\end{tikzpicture}
$$
is simultaneously $(3,7,8,5,6,1,4,9,2)$-atomic, $(6,1,4,9,2,5,3,7,8)$-decomposable, and $\UIO_9$-incompatible.

\subsection{Power-sum friendly posets} 

We now give a more general lemma concerning posets of set partitions with linear order.  If $\leq$ is a partial order on ordered set partitions, let $\leq_\phi$ denote the restriction of the order to $\cS^\phi$.  A partial order $\leq$ on ordered set partitions is \emph{power-sum friendly} if
\begin{enumerate}
\item Let $\phi\in L[I]$, $\tau\in L[J]$ with $I\cap J=\emptyset$, and $\rho\in L[I\cup J]$ with $\rho|_I=\phi$ and $\rho|_J=\tau$.  For all $\lambda\in \cS^\phi$, $\nu\in \cS^\tau$, and $\mu\in \cS^\rho$,
$$\mu_{I}\geq_\phi \lambda \text{ and } \mu_{J}\geq_\tau \nu\qquad \text{if and only if} \qquad \mu\geq_\rho \lambda\cup \nu,$$
\item Let $\phi,\tau\in L[K]$.  Then $\lambda\in \cS^\phi-\cS^\tau$ implies $\mu\in\cS^\phi- \cS^\tau$ for all $\mu\geq_\phi \lambda$.
\end{enumerate}

Define 
$$P^{\geq}_{(\phi,\lambda)}=\sum_{\mu\geq_\phi \lambda}\kappa_{(\phi,\mu)}.$$

\begin{example} \label{ex:powerfriendly}
There are three natural examples that are power-sum friendly.
\begin{enumerate}
\item[(a)] We can order set partitions in the refinement order.  That is, $\mu\geq \lambda$ if $i,j$ in the same part of $\lambda$ implies $i,j$ are in the same part of $\mu$ (here, set partitions are equivalence relations on sets).  This relation gives power-sums as defined in~\cite{RS,BZ}.
\item[(b)] We can define $\mu\geq \lambda$ if $\mu\supseteq \lambda$ (here, set partitions are sets of arcs). In the case where $q=2$, this relation gives the $\bf q$ basis as defined in Section 3 of~\cite{BZ}.
\item[(c)] We can define $\mu\geq \lambda$ if $\mu\succeq \lambda$, where $\succeq$ is defined in Section~\ref{sec:Mbasis} (here, set partitions are sets of arcs).
\end{enumerate}
\end{example}

\begin{lemma}\label{lem:Pbasis}
If $\geq$ is power-sum friendly, then 
\begin{align*}
P^{\geq}_{(\phi,\lambda)}P^{\geq}_{(\tau,\nu)}&=P_{(\phi\tau,\lambda\cup \nu)}^\geq, &&  \lambda\in \cS^\phi, \nu\in \cS^\tau,\\
\Delta(P^\geq_{(\phi,\lambda)})&=\sum_{I\sqcup J=K\atop \lambda=\lambda_I\cup\lambda_J} P^\geq_{(\phi|_I,\lambda_I)}\otimes P^\geq_{(\phi|_J,\lambda_J)} && \lambda\in \cS^\phi.
\end{align*}
\end{lemma}
\begin{proof}
Note that 
\begin{align*}
P^{\geq}_{(\phi,\lambda)}P^{\geq}_{(\tau,\nu)}&=\sum_{\mu\geq \lambda\atop \gamma\geq \nu}\kappa_{(\phi,\mu)}\kappa_{(\phi,\gamma)}
=\sum_{\mu\geq \lambda\atop \gamma\geq \nu}\sum_{\delta|_I=\mu\atop \delta|_J=\gamma} \kappa_{(\phi\tau,\delta)}\\
&=\sum_{\delta|_I\geq \lambda\atop \delta|_J\geq \nu} \kappa_{(\phi\tau,\delta)}
=\sum_{\delta\geq \lambda\cup \nu} \kappa_{(\phi\tau,\delta)}=P^\geq_{(\phi\tau,\lambda\cup \nu)},
\end{align*}
where the second to last equality follows from condition (1) of power-sum friendly.
For the coproduct, condition (2) of power-sum friendly implies the third equality in
\begin{align*}
\Delta(P^\geq_{(\phi,\lambda)}) &= \sum_{\mu\geq \lambda} \Delta(\kappa_{(\phi,\mu)})
=\sum_{\mu\geq \lambda}\sum_{K=I\sqcup J\atop \mu=\mu_I\cup\mu_J} \kappa_{(\phi|_I,\mu_I)}\otimes \kappa_{(\phi|_J,\mu_J)}\\
&=\sum_{K=I\sqcup J\atop \lambda=\lambda_I\cup\lambda_J}\sum_{\mu\geq \lambda_I\in \cS^{\phi|_I}\atop \nu\geq \lambda_J\in  \cS^{\phi|_J}} \kappa_{(\phi|_I,\mu)}\otimes \kappa_{(\phi|_J,\nu)}\\
& =\sum_{K=I\sqcup J\atop \lambda=\lambda_I\cup\lambda_J} P^\geq_{(\phi|_I,\lambda_I)}\otimes P^\geq_{(\phi|_J,\lambda_J)},
\end{align*}
as desired.
\end{proof}

The multiplication rule of the $P^{\geq}$ gives directly the following.

\begin{corollary} \label{cor:Patomic}
${{\bf scf}(U)}$ is freely generated by the $P^{\geq}_{(\phi,\lambda)}$ where $(\phi,\lambda)$ is atomic.
\end{corollary}

\subsection{$q$-Analogues of power sums}

For $\lambda\in \cS^\phi$, define
$$P_{(\phi,\lambda)}^{(q)}=\sum_{\mu\supseteq \lambda} \frac{1}{q^{\nst^\lambda_{\mu-\lambda}}} \kappa_{(\phi,\mu)}.$$
Note that at $q=1$, these functions reduce to the the basis of the previous section coming from the power-sum friendly relation $\subseteq$ on set partitions.  In \cite{BT} these $q$-analogues come up naturally in the representation theory of $U_n$, giving  a upper-lower triangular decomposition of the supercharacter table of $U_n$.

Note that the product of these functions is easy to compute, since the nesting functions add nicely.  Thus, for $\lambda\in \cS^\phi$ and $\mu\in \cS^\tau$,
$$
P_{(\phi,\lambda)}^{(q)} P_{(\tau,\mu)}^{(q)}  = P_{(\phi\tau, \lambda\cup \mu)}^{(q)}.
$$
However, the coproduct is more interesting.  We first prove a relevant factorization formula.

\begin{lemma} \label{RestrictionFactorization} Let $\phi\in L[K]$, $J\subseteq K$, $\lambda\in \cS^\phi$, and $\mu\in \cS^{\phi|_J}$ with $\mu\supseteq \lambda_J$.  Then
$$\sum_{\nu\in \cS^{\phi|_J} \atop \mu\supseteq \nu\supseteq \lambda_J} \frac{(-1)^{|\mu-\nu|}}{q^{\nst_{\mu-\nu}^\nu+\nst_{\nu-\lambda_J}^\lambda}}=\prod_{i\slarc{}j\in \mu-\lambda_J} \bigg(\frac{1}{q^{\nst_{i\slarc{}j}^\lambda}}- \frac{1}{q^{\nst_{i\slarc{}j}^\mu}}\bigg).$$
\end{lemma}

\begin{proof}
For a fixed $\nu\in  \cS^{\phi|_J}$ with $\mu\supseteq \nu\supseteq \lambda_J$, note that
$$\frac{(-1)^{|\mu-\nu|}}{q^{\nst_{\mu-\nu}^\nu+\nst_{\nu-\lambda_J}^\lambda}}=\prod_{i\slarc{}j\in \nu-\lambda_J} \frac{1}{q^{\nst_{i\slarc{}j}^\lambda}} \prod_{i\slarc{}j\in \mu-\nu} \frac{-1}{q^{\nst_{i\slarc{}j}^\nu}}.$$
Thus,
\begin{align*}
\prod_{i\slarc{}j\in \mu-\lambda_J} \bigg(\frac{1}{q^{\nst_{i\slarc{}j}^\lambda}}- \frac{1}{q^{\nst_{i\slarc{}j}^\mu}}\bigg)&=\sum_{\mu\supseteq \nu\supseteq \lambda|_J} \prod_{i\slarc{}j\in \nu} \frac{1}{q^{\nst_{i\slarc{}j}^\lambda}} \prod_{i\slarc{}j\in \mu-\nu} \frac{-1}{q^{\nst_{i\slarc{}j}^\nu}}\\
&=\sum_{\nu\in \cS^{\phi|_J} \atop \mu\supseteq \nu\supseteq \lambda_J} \frac{(-1)^{|\mu-\nu|}}{q^{\nst_{\mu-\nu}^\nu+\nst_{\nu-\lambda_J}^\lambda}},
\end{align*}
as desired.
\end{proof}

To obtain the antipode below, we want to use Takeuchi's formula (\ref{TakeuchisFormula}), so we compute the restriction rule to arbitrary parabolic subgroups.

\begin{theorem}\label{PRestrictionTheorem}
For $\phi\in L[K]$, $\lambda\in \cS^\phi$ and $J=(J_1,\ldots, J_\ell)\models K$,
$$\Delta_{J}(P_{(\phi,\lambda)}^{(q)}) = \left\{\begin{array}{ll} 
\displaystyle \sum_{\mu\in \cS^\phi,\mu\supseteq \lambda \atop \mu=\mu_{J_1}\cup\cdots \cup \mu_{J_\ell}}\hspace{-.5cm} a_{J,\mu}^\lambda P_{(\phi|_{J_1},\mu_{J_1})}^{(q)}\otimes\cdots \otimes P_{(\phi|_{J_\ell},\mu_{J_\ell})}^{(q)}
 & \text{if $\lambda=\lambda_{J_1}\cup\cdots \cup \lambda_{J_\ell}$,}  \\ 0 & \text{otherwise},\end{array}\right.
$$
where 
$$a_{J,\mu}^\lambda=\prod_{k=1}^\ell \prod_{i\slarc{}j\in \mu_{J_k}-\lambda}\bigg(\frac{1}{q^{\nst_{i\slarc{} j}^\lambda}}-\frac{1}{q^{\nst_{i\slarc{} j}^{\mu_{J_k}}}}\bigg).$$
\end{theorem}

\begin{proof}
By definition,
\begin{align*}
\Delta_{J}(P_{(\phi,\lambda)}^{(q)}) &= \sum_{\nu\in \cS^\phi, \nu\supseteq \lambda} \frac{1}{q^{\nst_{\nu-\lambda}^\lambda}}
\Delta_{J}(\kappa_{(\phi,\nu)})\\
&=\sum_{\nu\in \cS^\phi, \nu\supseteq \lambda\atop
\nu=\nu_{J_1}\cup\cdots \cup \nu_{J_\ell}} \frac{1}{q^{\nst_{\nu-\lambda}^\lambda}}
\kappa_{(\phi|_{J_1},\nu_{J_1})}\otimes \cdots \otimes \kappa_{(\phi|_{J_\ell},\nu_{J_\ell})}.
\end{align*}
Thus, if $\lambda\neq \lambda_{J_1}\cup\lambda_{J_2}\cup\cdots \cup \lambda_{J_\ell}$, then $\Delta_J(P_{(\phi,\lambda)}^{(q)})=0$.  Assume that $\lambda= \lambda_{J_1}\cup\lambda_{J_2}\cup\cdots \cup \lambda_{J_\ell}$.  Then we can write
\begin{align*}
\Delta_{J}(P_{(\phi,\lambda)}^{(q)}) &=\sum_{{\nu_k\in \cS^{\phi|_{J_k}}\atop \nu_k\supseteq \lambda_{J_k}}\atop 1\leq k\leq \ell} \Big(\frac{1}{q^{\nst_{\nu_1-\lambda}^\lambda}}
\kappa_{(\phi|_{J_1},\nu_1)}\Big)\otimes \cdots \otimes \Big(\frac{1}{q^{\nst_{\nu_\ell-\lambda}^\lambda}}\kappa_{(\phi|_{J_\ell},\nu_\ell)}\Big)\\
&=\bigotimes_{k=1}^\ell \bigg(\sum_{\nu_k\in \cS^{\phi|_{J_k}}\atop \nu_{k}\supseteq \lambda_{J_k}} \frac{1}{q^{\nst_{\nu_{k}-\lambda}^\lambda}}
\kappa_{(\phi|_{J_k},\nu_{k})}\bigg).
\end{align*}
By \cite[Proposition 3.1]{BT},
$$\kappa_{(\phi,\nu)}=\sum_{\mu\in \cS^\phi, \mu\supseteq \nu} \frac{(-1)^{|\mu-\nu|}}{q^{\nst_{\mu-\nu}^{\mu}}}P_{(\phi,\mu)}^{(q)},$$
so
\begin{align*}
\Delta_{J}(P_{(\phi,\lambda)}^{(q)})&=\bigotimes_{k=1}^\ell \bigg(\sum_{\mu_k,\nu_k\in \cS^{\phi|_{J_k}}\atop \mu_k\supseteq \nu_{k}\supseteq \lambda_{J_k}} \frac{(-1)^{|\mu_k-\nu_k|}}{q^{\nst_{\nu_{k}-\lambda}^\lambda+\nst_{\mu_k-\nu_k}^{\mu_k}}}
P_{(\phi|_{J_k},\mu_{k})}^{(q)}\bigg)\\
&=\bigotimes_{k=1}^\ell \bigg(\sum_{\mu_k\in\cS^{\phi|_{J_k}}\atop \mu_k\supseteq \lambda_{J_k}}\bigg(\sum_{\nu_k\in \cS^{\phi|_{J_k}}\atop \mu_k\supseteq \nu_{k}\supseteq \lambda_{J_k}}\frac{(-1)^{|\mu_k-\nu_k|}}{q^{\nst_{\nu_{k}-\lambda}^\lambda+\nst_{\mu_k-\nu_k}^{\mu_k}}}
\bigg)P_{(\phi|_{J_k},\mu_{k})}^{(q)}\bigg).
\end{align*}
By Lemma \ref{RestrictionFactorization},
\begin{align*}
\Delta_{J}(P_{(\phi,\lambda)}^{(q)})&=\bigotimes_{k=1}^\ell \bigg(\sum_{\mu_k\in\cS^{\phi|_{J_k}}\atop \mu_k\supseteq \lambda_{J_k}}\prod_{i\slarc{}j\in\mu_k-\lambda_{J_k}}\bigg(\frac{1}{q^{\nst_{i\slarc{}j}^\lambda}}-\frac{1}{q^{\nst_{i\slarc{}j}^{\mu_k}}}
\bigg)P_{(\phi|_{J_k},\mu_{k})}^{(q)}\bigg)\\
&= \sum_{\mu\in \cS^\phi,\mu\supseteq \lambda \atop \mu=\mu_{J_1}\cup\cdots \cup \mu_{J_\ell}}\hspace{-.5cm} a_{J,\mu}^\lambda P_{(\phi|_{J_1},\mu_{J_1})}^{(q)}\otimes\cdots \otimes P_{(\phi|_{J_\ell},\mu_{J_\ell})}^{(q)},
\end{align*}
as desired.
\end{proof}

\begin{remark}
The coefficient $a_{J,\mu}^\lambda$ is not always nonzero.  For example if there exists $j\larc{}k\in \mu-\lambda$ such that $\nst_{j\slarc{}k}^\lambda=0$, then $a_{J,\mu}^\lambda=0$.
\end{remark}

Using Takeuchi's formula, we now obtain a formula for the antipode.

\begin{corollary} \label{qPAntipode} For $\phi\in L[K]$, and $\lambda\in S^\phi$,
$$S(P_{(\phi,\lambda)}^{(q)})=\sum_{\tau\in L[K], \mu\in \cS^\tau\atop \mu\supset \lambda, \mu\ \phi\text{-atomic}} (-1)^{D_\phi(\tau)} \prod_{i\slarc{}j\in \mu-\lambda}\bigg( \frac{1}{q^{\nst_{i\slarc{}j}^\lambda}}- \frac{1}{q^{\nst_{i\slarc{}j}^\mu}}\bigg) P_{(\tau,\mu)}^{(q)},$$
where $D_\phi(\tau)$ is the number of descents of $\tau$ with respect to $\phi$.
\end{corollary}
\begin{proof}
 By (\ref{TakeuchisFormula}) and Theorem \ref{PRestrictionTheorem},
\begin{align*}
S(P_{(\phi,\lambda)}^{(q)}) &= \sum_{J=(J_1,J_2,\ldots, J_\ell)} (-1)^\ell m_J\bigg(\sum_{\mu\supseteq \lambda\atop \mu=\mu_{J_1}\cup\mu_{J_2}\cup \cdots \cup \mu_{J_\ell} } \hspace{-.8cm} a_{J,\mu}^\lambda P_{(\phi|_{J_1}, \mu_{J_1})}^{(q)}\otimes \cdots \otimes P_{(\phi|_{J_\ell}, \mu_{J_\ell})}^{(q)} \bigg)\\
&= \sum_{J=(J_1,J_2,\ldots, J_\ell)\atop \mu=\mu_{J_1}\cup\mu_{J_2}\cup \cdots \cup \mu_{J_\ell}\supseteq \lambda} (-1)^\ell  a_{J,\mu}^\lambda P_{(\phi|_{J_1}\phi|_{J_2}\cdots \phi|_{J_\ell}, \mu)}^{(q)}\\
&= \sum_{\tau\in L[K], \lambda\in \cS^\tau\atop \mu\in \cS^\tau,\mu\supseteq \lambda}\bigg(\sum_{{J=(J_1,J_2,\ldots, J_\ell)\atop \mu=\mu_{J_1}\cup\mu_{J_2}\cup \cdots \cup \mu_{J_\ell}}\atop \tau=\phi|_{J_1}\phi|_{J_2}\cdots\phi|_{J_\ell}} (-1)^\ell  a_{J,\mu}^\lambda \bigg)P_{(\tau, \mu)}^{(q)}.
\end{align*}
Note that $a_{J,\mu}^\lambda=a_{J',\mu}^\lambda$ if $\mu=\mu_{J_1}\cup\mu_{J_2}\cup\cdots \cup \mu_{J_\ell}=\mu_{J'_1}\cup\mu_{J'_2}\cup\cdots \cup \mu_{J'_{\ell'}}$ and $\phi|_{J_1}\phi|_{J_2}\cdots\phi|_{J_\ell}=\phi|_{J'_1}\phi|_{J'_2}\cdots\phi|_{J'_{\ell'}}$, so
$$
S(P_{(\phi,\lambda)}^{(q)}) = \sum_{\tau\in L[K], \lambda\in \cS^\tau\atop \mu\in \cS^\tau,\mu\supseteq \lambda}\bigg(\sum_{{J=(J_1,J_2,\ldots, J_\ell)\atop \mu=\mu_{J_1}\cup\mu_{J_2}\cup \cdots \cup \mu_{J_\ell}}\atop \tau=\phi|_{J_1}\phi|_{J_2}\cdots\phi|_{J_\ell}} (-1)^\ell  \bigg)a_{J,\mu}^\lambda  P_{(\tau, \mu)}^{(q)}.
$$

Fix a $\phi$-compatible $\mu\in \cS^\tau$.   Note that $\tau$ has  at least the following two factorizations (which could coincide).
\begin{itemize}
\item $\tau=\tau_1\tau_2\ldots \tau_l$ into maximal rising subsequences with respect to $\phi$,
\item $\tau=\tau'_1\tau'_2\ldots\tau'_L$ where $L$ is maximal such that each $\tau'_j$ is a rising subsequence with respect to $\phi$ and $\mu=\mu|_{\tau'_1}\cup\cdots \cup\mu|_{\tau'_{L}}$.
\end{itemize}
If $C$ is the set of positions where new factors start in the first factorization and $F$ is the set of positions where the new factors start in the second, then every factorization of $\tau$ into rising sequences that respect the arcs of $\mu$ have positions $P$ with $C\subseteq P\subseteq F$.  Thus,
\begin{align*}
\sum_{{J=(J_1,J_2,\ldots, J_\ell)\atop \mu=\mu_{J_1}\cup\mu_{J_2}\cup \cdots \cup \mu_{J_\ell}}\atop \tau=\phi|_{J_1}\phi|_{J_2}\cdots\phi|_{J_\ell}} (-1)^\ell & = \sum_{C\subseteq P\subseteq F} (-1)^{|P|}\\
& = (-1)^{|C|}\sum_{C\subseteq P\subseteq F} \mathrm{mb}(P,C).
\end{align*}
where $ \mathrm{mb}(P,C)$ is the M\"obius function of the subsets of $F$ ordered by inclusion.  Thus,
$$\sum_{{J=(J_1,J_2,\ldots, J_\ell)\atop \mu=\mu_{J_1}\cup\mu_{J_2}\cup \cdots \cup \mu_{J_\ell}}\atop \tau=\phi|_{J_1}\phi|_{J_2}\cdots\phi|_{J_\ell}} (-1)^\ell =\left\{\begin{array}{ll} (-1)^{L} & \text{if $l=L$,}\\ 0 & \text{otherwise.}\end{array}\right.$$

The condition that $l=L$ is equivalent to the condition that $(\tau,\mu)$ is $\phi$-atomic, and, in this case, $L=D_\phi(\tau)$.   Thus,
$$S(P_{(\phi,\lambda)}^{(q)}) = \sum_{{\tau\in L[K], \lambda\in \cS^\tau\atop \mu\in \cS^\tau,\mu\supseteq \lambda}\atop \mu\text{ $\phi$-atomic}} (-1)^{D_\phi(\tau)}a_{J,\mu}^\lambda  P_{(\tau, \mu)}^{(q)},$$
as desired.
\end{proof}

Note 
$$P_{(\phi,\lambda)}^{(1)}=P_{(\phi,\lambda)}^\subseteq$$
is a function corresponding to a power sum friendly relation.  The antipode does not care which power-sum friendly relation we choose, so we obtain the following corollary by setting $q=1$ in Corollary \ref{qPAntipode}.

\begin{corollary}\label{cor:P}
For a finite set $K$, $\phi\in L[K]$, $\lambda\in \cS^\phi$, and $\geq$ a power-sum friendly relation, 
$$ S(P^\geq_{(\phi,\lambda)}) = \sum_{\tau\in L[K],\lambda\in \cS^\tau \atop \text{$\lambda$ $\phi$-atomic}} 
                                                            (-1)^{ D_\phi(\tau) }P^\geq_{(\tau, \lambda)}.$$
\end{corollary}

The map $\kappa_{(\phi,\lambda)}\mapsto \kappa_\lambda$ described in Remark~\ref{rem:functor} defines  
$P^{\ge}_\lambda\in\overline{\mathcal K}({{\bf scf}(U)})$ from $P^{\ge}_{(\phi,\lambda)}\in\SC{}{}{}$. Recall that $\kappa_\lambda$ and $P^{\ge}_\lambda$ are indexed by sets of arcs $\lambda$ that are in $\bigcup_{n\ge 0}\cS^{\UIO_n}$.
We obtain the following direct consequence of Corollary~\ref{cor:P}.

\begin{corollary}\label{cor2:P} In the Hopf algebra $\overline{\mathcal K}({{\bf scf}(U)})$, for $n\ge 0$ and  $\lambda\in \cS^{\UIO_n}$, we have
 $$ S(P^{\ge}_\lambda) = \sum_{\mu\in \cS^{\UIO_n}} (-1)^{\ell(\mu)} c_{\lambda,\mu} P^{\ge}_\mu,
 $$
where $\ell(\mu)$ is the number of components in the atomic decomposition of $(\UIO_n,\mu)$ and 
 $$ c_{\lambda,\mu} = \# \{ \tau\in L[n] \mid \lambda\in \cS^\tau \hbox{ is $\UIO$-atomic},\, \mu=\tau^{-1}\circ\lambda \}.$$
\end{corollary}
An equivalent combinatorial description for the constants $c_{\lambda,\mu}$ is as follow. Given $\mu$, 
$$c_{\lambda,\mu}=\#\left\{w\in \fkS_n\ \bigg|\ \begin{array}{@{}l@{}} w\circ \mu=\lambda,   \text{ $w(j)>w(j+1)$ if and}\\ \text{only if $i\larc{}k\notin \mu$ for $i\leq j<j+1\leq k$}\end{array}\right\}.$$
For example, if
 $$ \lambda =
 \begin{tikzpicture}[baseline=.2cm]
	\foreach \x in {1,2,3,4,5} 
		\node (\x) at (\x/2,0) [inner sep=-1pt] {$\bullet$};
	\node at (1/2,-.2) {$\scriptstyle 1$};
	\node at (2/2,-.2) {$\scriptstyle 2$};
	\node at (3/2,-.2) {$\scriptstyle 3$};
	\node at (4/2,-.2) {$\scriptstyle 4$};
	\node at (5/2,-.2) {$\scriptstyle 5$};
	\draw (2) .. controls (3/2,.75) and (4/2,.75) ..  node [above=-2pt] { } (5); 
	\draw (1) .. controls (1.25/2,.5) and (1.75/2,.5) ..  node [above=-2pt] { } (2); 
\end{tikzpicture} ,\qquad
\mu =
 \begin{tikzpicture}[baseline=.2cm]
	\foreach \x in {1,2,3,4} 
		\node (\x) at (\x/2,0) [inner sep=-1pt] {$\bullet$};
	\node at (1/2,-.2) {$\scriptstyle 1$};
	\node at (2/2,-.2) {$\scriptstyle 2$};
	\node at (3/2,-.2) {$\scriptstyle 3$};
	\node at (4/2,-.2) {$\scriptstyle 4$};
	\draw (2) .. controls (2.5/2,.75) and (3.5/2,.75) ..  node [above=-2pt] { } (4); 
	\draw (1) .. controls (1.25/2,.5) and (1.75/2,.5) ..  node [above=-2pt] { } (2); 
\end{tikzpicture} \Big|
 \begin{tikzpicture}[baseline=.2cm]
	\foreach \x in {1} 
		\node (\x) at (\x/2,0) [inner sep=-1pt] {$\bullet$};
	\node at (1/2,-.2) {$\scriptstyle 5$};
\end{tikzpicture} 
$$
then
 $$c_{\lambda,\mu}=\#\{(1,2,3,5,4),(1,2,4,5,3)\}=2.$$ 

Define a partial order $\lhd$ on $\bigcup_{\phi\in L[K]}\cS^\phi$ by $(\phi,\nu)\lhd (\tau,\lambda)$ if $|\lambda|<|\nu|$ or $|\lambda|=|\mu|$ with $\dim(\tau,\lambda)>\dim(\phi,\nu)$.
 
\begin{remark} The formula in Corollary~\ref{cor:P} is multiplicity free and in particular cancelation free. With respect to $\lhd$ 
 the largest term in the antipode $S(P_{(\phi,\lambda)})$ is $(-1)^{k}P_{(\phi_k\cdots\phi_2\phi_1,\lambda)}$.
The formula in Corollary~\ref{cor2:P} is cancelation free and is valid for any power-friendly order.
In particular, for the example~\ref{ex:powerfriendly} (a), the formula we get is an improvement on the formula of~\cite{LM}.
\end{remark}

\section{Antipode on the superclass characteristic functions}\label{sec:Mbasis}

It is interesting to compute the antipode formula for the two other natural basis of ${{\bf scf}(U)}$. We consider in this section the superclass characteristic functions.

Let $\lambda\in\cS^\phi\cap \cS^\tau$ for $\phi,\tau\in L[K]$, and suppose $\phi=\phi_1\phi_2\cdots\phi_k$ is the unique factorization into maximal rising subsequences with respect to $\tau$.  Let
$$A_\tau(\lambda,\phi)=\left\{\mu\in \cS^\phi\ \bigg|\  \begin{array}{@{}l@{}}\text{ $\mu$ $\tau$-atomic with $\mu\supseteq \lambda$, and}\\ \text{$\mu\supset \nu\supseteq \lambda$ implies $\nu$ $\tau$-decomposable}\end{array}\right\},$$
or the set of minimal coarsenings of $\lambda$ to $\tau$-atomic set-partitions.

For example, with
  $$
(\phi,\lambda)=\begin{tikzpicture}[baseline=.2cm]
	\foreach \x in {1,2,3,4,5,6,7,8} 
		\node (\x) at (\x/2,0) [inner sep=-1pt] {$\bullet$};
	\node at (1/2,-.2) {$\scriptstyle 2$};
	\node at (2/2,-.2) {$\scriptstyle 4$};
	\node at (3/2,-.2) {$\scriptstyle 5$};
	\node at (4/2,-.2) {$\scriptstyle 6$};
	\node at (5/2,-.2) {$\scriptstyle 1$};
	\node at (6/2,-.2) {$\scriptstyle 3$};
	\node at (7/2,-.2) {$\scriptstyle 7$};
	\node at (8/2,-.2) {$\scriptstyle 8$};
	\draw (2) .. controls (2.5/2,.75) and (3.5/2,.75) ..  node [above=-2pt] { } (4); 
\end{tikzpicture} 
$$
The set $A_{12345678}(\phi,\lambda)$ is the Cartesian product of the two sets
 $$A_{2456}\Big(
 \begin{tikzpicture}[baseline=.2cm]
	\foreach \x in {1,2,3,4} 
		\node (\x) at (\x/2,0) [inner sep=-1pt] {$\bullet$};
	\node at (1/2,-.2) {$\scriptstyle 2$};
	\node at (2/2,-.2) {$\scriptstyle 4$};
	\node at (3/2,-.2) {$\scriptstyle 5$};
	\node at (4/2,-.2) {$\scriptstyle 6$};
	\draw (2) .. controls (2.5/2,.75) and (3.5/2,.75) ..  node [above=-2pt] { } (4); 
\end{tikzpicture} 
\Big) = \Big\{
 \begin{tikzpicture}[baseline=.2cm]
	\foreach \x in {1,2,3,4} 
		\node (\x) at (\x/2,0) [inner sep=-1pt] {$\bullet$};
	\node at (1/2,-.2) {$\scriptstyle 2$};
	\node at (2/2,-.2) {$\scriptstyle 4$};
	\node at (3/2,-.2) {$\scriptstyle 5$};
	\node at (4/2,-.2) {$\scriptstyle 6$};
	\draw (2) .. controls (2.5/2,.75) and (3.5/2,.75) ..  node [above=-2pt] { } (4); 
	\draw [densely dotted] (1) .. controls (1.25/2,.5) and (1.75/2,.5) ..  node [above=-2pt] { } (2); 
\end{tikzpicture} 
,\ 
 \begin{tikzpicture}[baseline=.2cm]
	\foreach \x in {1,2,3,4} 
		\node (\x) at (\x/2,0) [inner sep=-1pt] {$\bullet$};
	\node at (1/2,-.2) {$\scriptstyle 2$};
	\node at (2/2,-.2) {$\scriptstyle 4$};
	\node at (3/2,-.2) {$\scriptstyle 5$};
	\node at (4/2,-.2) {$\scriptstyle 6$};
	\draw (2) .. controls (2.5/2,.75) and (3.5/2,.75) ..  node [above=-2pt] { } (4); 
	\draw [densely dotted] (1) .. controls (1.5/2,.75) and (2.5/2,.75) ..  node [above=-2pt] { } (3); 
\end{tikzpicture} \Big\}
$$
and
$$A_{1378}\Big(
 \begin{tikzpicture}[baseline=.2cm]
	\foreach \x in {1,2,3,4} 
		\node (\x) at (\x/2,0) [inner sep=-1pt] {$\bullet$};
	\node at (1/2,-.2) {$\scriptstyle 1$};
	\node at (2/2,-.2) {$\scriptstyle 3$};
	\node at (3/2,-.2) {$\scriptstyle 7$};
	\node at (4/2,-.2) {$\scriptstyle 8$};
\end{tikzpicture} 
\Big) = \Big\{
 \begin{tikzpicture}[baseline=.2cm]
	\foreach \x in {1,2,3,4} 
		\node (\x) at (\x/2,0) [inner sep=-1pt] {$\bullet$};
	\node at (1/2,-.2) {$\scriptstyle 1$};
	\node at (2/2,-.2) {$\scriptstyle 3$};
	\node at (3/2,-.2) {$\scriptstyle 7$};
	\node at (4/2,-.2) {$\scriptstyle 8$};
	\draw [densely dotted] (1) .. controls (1.5/2,1) and (3.5/2,1) ..  node [above=-2pt] { } (4); 
\end{tikzpicture} 
,\ 
 \begin{tikzpicture}[baseline=.2cm]
	\foreach \x in {1,2,3,4} 
		\node (\x) at (\x/2,0) [inner sep=-1pt] {$\bullet$};
	\node at (1/2,-.2) {$\scriptstyle 1$};
	\node at (2/2,-.2) {$\scriptstyle 3$};
	\node at (3/2,-.2) {$\scriptstyle 7$};
	\node at (4/2,-.2) {$\scriptstyle 8$};
	\draw [densely dotted] (1) .. controls (1.5/2,.75) and (2.5/2,.75) ..  node [above=-2pt] { } (3); 
	\draw [densely dotted] (3) .. controls (3.25/2,.5) and (3.75/2,.5) ..  node [above=-2pt] { } (4); 
\end{tikzpicture} 
,\ 
 \begin{tikzpicture}[baseline=.2cm]
	\foreach \x in {1,2,3,4} 
		\node (\x) at (\x/2,0) [inner sep=-1pt] {$\bullet$};
	\node at (1/2,-.2) {$\scriptstyle 1$};
	\node at (2/2,-.2) {$\scriptstyle 3$};
	\node at (3/2,-.2) {$\scriptstyle 7$};
	\node at (4/2,-.2) {$\scriptstyle 8$};
	\draw [densely dotted] (1) .. controls (1.5/2,.75) and (2.5/2,.75) ..  node [above=-2pt] { } (3); 
	\draw [densely dotted] (2) .. controls (2.5/2,.75) and (3.5/2,.75) ..  node [above=-2pt] { } (4); 
\end{tikzpicture} ,
$$
$$\hfill\hfill\ \qquad
 \begin{tikzpicture}[baseline=.2cm]
	\foreach \x in {1,2,3,4} 
		\node (\x) at (\x/2,0) [inner sep=-1pt] {$\bullet$};
	\node at (1/2,-.2) {$\scriptstyle 1$};
	\node at (2/2,-.2) {$\scriptstyle 3$};
	\node at (3/2,-.2) {$\scriptstyle 7$};
	\node at (4/2,-.2) {$\scriptstyle 8$};
	\draw [densely dotted] (2) .. controls (2.5/2,.75) and (3.5/2,.75) ..  node [above=-2pt] { } (4); 
	\draw [densely dotted] (1) .. controls (1.25/2,.5) and (1.75/2,.5) ..  node [above=-2pt] { } (2); 
\end{tikzpicture} 
,\ 
 \begin{tikzpicture}[baseline=.2cm]
	\foreach \x in {1,2,3,4} 
		\node (\x) at (\x/2,0) [inner sep=-1pt] {$\bullet$};
	\node at (1/2,-.2) {$\scriptstyle 1$};
	\node at (2/2,-.2) {$\scriptstyle 3$};
	\node at (3/2,-.2) {$\scriptstyle 7$};
	\node at (4/2,-.2) {$\scriptstyle 8$};
	\draw [densely dotted] (1) .. controls (1.25/2,.5) and (1.75/2,.5) ..  node [above=-2pt] { } (2); 
	\draw [densely dotted] (2) .. controls (2.25/2,.5) and (2.75/2,.5) ..  node [above=-2pt] { } (3); 
	\draw [densely dotted] (3) .. controls (3.25/2,.5) and (3.75/2,.5) ..  node [above=-2pt] { } (4); 
\end{tikzpicture} 
\Big\}.
$$
Note that 
$$ 
 \begin{tikzpicture}[baseline=.2cm]
	\foreach \x in {1,2,3,4} 
		\node (\x) at (\x/2,0) [inner sep=-1pt] {$\bullet$};
	\node at (1/2,-.2) {$\scriptstyle 1$};
	\node at (2/2,-.2) {$\scriptstyle 3$};
	\node at (3/2,-.2) {$\scriptstyle 7$};
	\node at (4/2,-.2) {$\scriptstyle 8$};
	\draw [densely dotted] (1) .. controls (1.5/2,1) and (3.5/2,1) ..  node [above=-2pt] { } (4); 
	\draw [densely dotted] (2) .. controls (2.25/2,.5) and (2.75/2,.5) ..  node [above=-2pt] { } (3); 
\end{tikzpicture} 
\not\in A_{1378}\Big(
 \begin{tikzpicture}[baseline=.2cm]
	\foreach \x in {1,2,3,4} 
		\node (\x) at (\x/2,0) [inner sep=-1pt] {$\bullet$};
	\node at (1/2,-.2) {$\scriptstyle 1$};
	\node at (2/2,-.2) {$\scriptstyle 3$};
	\node at (3/2,-.2) {$\scriptstyle 7$};
	\node at (4/2,-.2) {$\scriptstyle 8$};
\end{tikzpicture} 
\Big),
$$ 
since
$$ \begin{tikzpicture}[baseline=.2cm]
	\foreach \x in {1,2,3,4} 
		\node (\x) at (\x/2,0) [inner sep=-1pt] {$\bullet$};
	\node at (1/2,-.2) {$\scriptstyle 1$};
	\node at (2/2,-.2) {$\scriptstyle 3$};
	\node at (3/2,-.2) {$\scriptstyle 7$};
	\node at (4/2,-.2) {$\scriptstyle 8$};
	\draw [densely dotted] (1) .. controls (1.5/2,1) and (3.5/2,1) ..  node [above=-2pt] { } (4); 
	\draw [densely dotted] (2) .. controls (2.25/2,.5) and (2.75/2,.5) ..  node [above=-2pt] { } (3); 
\end{tikzpicture} \supseteq  \begin{tikzpicture}[baseline=.2cm]
	\foreach \x in {1,2,3,4} 
		\node (\x) at (\x/2,0) [inner sep=-1pt] {$\bullet$};
	\node at (1/2,-.2) {$\scriptstyle 1$};
	\node at (2/2,-.2) {$\scriptstyle 3$};
	\node at (3/2,-.2) {$\scriptstyle 7$};
	\node at (4/2,-.2) {$\scriptstyle 8$};
	\draw [densely dotted] (1) .. controls (1.5/2,1) and (3.5/2,1) ..  node [above=-2pt] { } (4); 
\end{tikzpicture} \supseteq  \begin{tikzpicture}[baseline=.2cm]
	\foreach \x in {1,2,3,4} 
		\node (\x) at (\x/2,0) [inner sep=-1pt] {$\bullet$};
	\node at (1/2,-.2) {$\scriptstyle 1$};
	\node at (2/2,-.2) {$\scriptstyle 3$};
	\node at (3/2,-.2) {$\scriptstyle 7$};
	\node at (4/2,-.2) {$\scriptstyle 8$};
\end{tikzpicture}
$$
and $\{1\larc{}8\}$ is not $1378$-decomposible.

\begin{theorem}\label{thm:M} For a finite set $K$, $\tau\in L[K]$, and for any $\lambda\in {\cS}^\tau$,
$$ S(\kappa_{(\tau,\lambda)}) = \sum_{\phi\in L[K] \atop \lambda\in \cS^\phi} 
                                                          \sum_{\mu\in A_{\tau}(\phi,\lambda)}  
                                                            (-1)^{|\mu| - | \lambda | + k }\kappa_{(\phi_1,\mu^{(1)})}  \cdots \kappa_{(\phi_k,\mu^{(k)})},$$
where $\phi=\phi_1\phi_2 \cdots\phi_k$ is the unique factorization of $\phi$ into maximal increasing subsequences with respect to $\tau$, and $\mu^{(i)}$ 
is the atomic set partition above $\phi_i$.
\end{theorem}

Before we prove the theorem, we want to understand the relationship between the summands in the theorem and the $\kappa$-basis.  Consider the following poset on $\cS^\phi$.  Let
  $$\lambda\prec\!\!\!\cdot\ \mu$$
if $\mu$ is obtained from $\lambda$ by adding an arc that connects two atomics of $\lambda$, where we recall that a cover $\nu\prec\!\!\!\cdot\ \mu$ is a relation such that the interval $[\nu,\mu]_\preceq =\{\nu,\mu\}$. 
Let $\preceq$ be the transitive closure of this relation.  For example, for $K=[3]$ and $\phi=123$,
$$
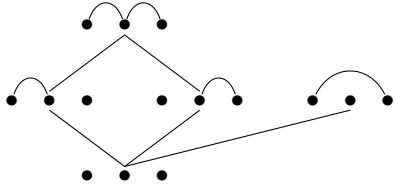
\begin{tikzpicture}[scale=.5]
	\foreach \x/\y in {0/0,1/0,2/0,2/-2,3/-2,4/-2,2/2,3/2,4/2,4/0,5/0,6/0,8/0,9/0,10/0}
		\node (\x\y) at (\x,\y) [inner sep = 0pt] {$\bullet$};
	\draw (00) .. controls (.25,.75) and (.75,.75) .. (10);
	\draw (50) .. controls (5.25,.75) and (5.75,.75) .. (60);
	\draw (80) .. controls (8.5,1) and (9.5,1) .. (100);
	\draw (22) .. controls (2.25,2.75) and (2.75,2.75) .. (32);
	\draw (32) .. controls (3.25,2.75) and (3.75,2.75) .. (42);
	\draw (3,-1.75) -- (1,-.25);
	\draw (3,-1.75) -- (5,-.25);
	\draw (3,-1.75) -- (9,-.25);
	\draw (1,.25) -- (3,1.75);
	\draw (5,.25) -- (3,1.75);
\end{tikzpicture}$$
Let $J=(J_1,\ldots, J_\ell)\models K$, $\phi_i\in L[J_i]$, and $\mu_i\in \cS^{\phi_i}$ atomic.   It follows from Proposition \ref{FunctorsSuperclasses} (2), that 
$$\kappa_{(\phi_1,\mu_1)}\cdots \kappa_{(\phi_\ell,\mu_\ell)} = \sum_{\nu\in \cS^{\phi_1\phi_2\cdots\phi_\ell}\atop \nu\succeq \mu_1\cup\mu_2\cdots\cup\mu_\ell} \kappa_{(\phi_1\phi_2\cdots\phi_\ell, \nu)}.$$
For example,
 $$
\kappa_{ \begin{tikzpicture}[baseline=.2cm]
	\foreach \x in {1,2,3,4} 
		\node (\x) at (\x/4,0) [inner sep=-1pt] {$\scriptstyle \bullet$};
	\node at (1/4,-.2) {$\scriptscriptstyle 2$};
	\node at (2/4,-.2) {$\scriptscriptstyle 4$};
	\node at (3/4,-.2) {$\scriptscriptstyle 5$};
	\node at (4/4,-.2) {$\scriptscriptstyle 6$};
	\draw (2) .. controls (2.5/4,.75/2) and (3.5/4,.75/2) ..  node [above=-2pt] { } (4); 
	\draw (1) .. controls (1.5/4,.75/2) and (2.5/4,.75/2) ..  node [above=-2pt] { } (3); 
\end{tikzpicture}  }
\kappa_{  \begin{tikzpicture}[baseline=.2cm]
	\foreach \x in {1,2,3,4} 
		\node (\x) at (\x/4,0) [inner sep=-1pt] {$\scriptstyle \bullet$};
	\node at (1/4,-.2) {$\scriptscriptstyle 1$};
	\node at (2/4,-.2) {$\scriptscriptstyle 3$};
	\node at (3/4,-.2) {$\scriptscriptstyle 7$};
	\node at (4/4,-.2) {$\scriptscriptstyle 8$};
	\draw  (2) .. controls (2.5/4,.75/2) and (3.5/4,.75/2) ..  node [above=-2pt] { } (4); 
	\draw  (1) .. controls (1.25/4,.5/2) and (1.75/4,.5/2) ..  node [above=-2pt] { } (2); 
\end{tikzpicture} }
=
\kappa_{  \begin{tikzpicture}[baseline=.2cm]
	\foreach \x in {1,2,3,4,5,6,7,8} 
		\node (\x) at (\x/4,0) [inner sep=-1pt] {$\scriptstyle \bullet$};
	\node at (1/4,-.2) {$\scriptscriptstyle 2$};
	\node at (2/4,-.2) {$\scriptscriptstyle 4$};
	\node at (3/4,-.2) {$\scriptscriptstyle 5$};
	\node at (4/4,-.2) {$\scriptscriptstyle 6$};
	\node at (5/4,-.2) {$\scriptscriptstyle 1$};
	\node at (6/4,-.2) {$\scriptscriptstyle 3$};
	\node at (7/4,-.2) {$\scriptscriptstyle 7$};
	\node at (8/4,-.2) {$\scriptscriptstyle 8$};
	\draw (2) .. controls (2.5/4,.75/2) and (3.5/4,.75/2) ..  node [above=-2pt] { } (4); 
	\draw (1) .. controls (1.5/4,.75/2) and (2.5/4,.75/2) ..  node [above=-2pt] { } (3); 
	\draw  (6) .. controls (6.5/4,.75/2) and (7.5/4,.75/2) ..  node [above=-2pt] { } (8); 
	\draw (5) .. controls (5.25/4,.5/2) and (5.75/4,.5/2) ..  node [above=-2pt] { } (6); 
\end{tikzpicture} }
+ \kappa_{  \begin{tikzpicture}[baseline=.2cm]
	\foreach \x in {1,2,3,4,5,6,7,8} 
		\node (\x) at (\x/4,0) [inner sep=-1pt] {$\scriptstyle \bullet$};
	\node at (1/4,-.2) {$\scriptscriptstyle 2$};
	\node at (2/4,-.2) {$\scriptscriptstyle 4$};
	\node at (3/4,-.2) {$\scriptscriptstyle 5$};
	\node at (4/4,-.2) {$\scriptscriptstyle 6$};
	\node at (5/4,-.2) {$\scriptscriptstyle 1$};
	\node at (6/4,-.2) {$\scriptscriptstyle 3$};
	\node at (7/4,-.2) {$\scriptscriptstyle 7$};
	\node at (8/4,-.2) {$\scriptscriptstyle 8$};
	\draw (2) .. controls (2.5/4,.75/2) and (3.5/4,.75/2) ..  node [above=-2pt] { } (4); 
	\draw (1) .. controls (1.5/4,.75/2) and (2.5/4,.75/2) ..  node [above=-2pt] { } (3); 
	\draw (6) .. controls (6.5/4,.75/2) and (7.5/4,.75/2) ..  node [above=-2pt] { } (8); 
	\draw (5) .. controls (5.25/4,.5/2) and (5.75/4,.5/2) ..  node [above=-2pt] { } (6); 
	\draw  [densely dotted]  (3) .. controls (3.5/4,.75/2) and (4.5/4,.75/2) ..  node [above=-2pt] { } (5); 
\end{tikzpicture} }
+ \kappa_{  \begin{tikzpicture}[baseline=.2cm]
	\foreach \x in {1,2,3,4,5,6,7,8} 
		\node (\x) at (\x/4,0) [inner sep=-1pt] {$\scriptstyle \bullet$};
	\node at (1/4,-.2) {$\scriptscriptstyle 2$};
	\node at (2/4,-.2) {$\scriptscriptstyle 4$};
	\node at (3/4,-.2) {$\scriptscriptstyle 5$};
	\node at (4/4,-.2) {$\scriptscriptstyle 6$};
	\node at (5/4,-.2) {$\scriptscriptstyle 1$};
	\node at (6/4,-.2) {$\scriptscriptstyle 3$};
	\node at (7/4,-.2) {$\scriptscriptstyle 7$};
	\node at (8/4,-.2) {$\scriptscriptstyle 8$};
	\draw (2) .. controls (2.5/4,.75/2) and (3.5/4,.75/2) ..  node [above=-2pt] { } (4); 
	\draw (1) .. controls (1.5/4,.75/2) and (2.5/4,.75/2) ..  node [above=-2pt] { } (3); 
	\draw (6) .. controls (6.5/4,.75/2) and (7.5/4,.75/2) ..  node [above=-2pt] { } (8); 
	\draw (5) .. controls (5.25/4,.5/2) and (5.75/4,.5/2) ..  node [above=-2pt] { } (6); 
	\draw  [densely dotted]  (3) .. controls (4/4,1/2) and (6/4,1/2) ..  node [above=-2pt] { } (7); 
\end{tikzpicture} } $$
$$\,  \hfill\hfill\hfill\hfill\hfill
+ \kappa_{  \begin{tikzpicture}[baseline=.2cm]
	\foreach \x in {1,2,3,4,5,6,7,8} 
		\node (\x) at (\x/4,0) [inner sep=-1pt] {$\scriptstyle \bullet$};
	\node at (1/4,-.2) {$\scriptscriptstyle 2$};
	\node at (2/4,-.2) {$\scriptscriptstyle 4$};
	\node at (3/4,-.2) {$\scriptscriptstyle 5$};
	\node at (4/4,-.2) {$\scriptscriptstyle 6$};
	\node at (5/4,-.2) {$\scriptscriptstyle 1$};
	\node at (6/4,-.2) {$\scriptscriptstyle 3$};
	\node at (7/4,-.2) {$\scriptscriptstyle 7$};
	\node at (8/4,-.2) {$\scriptscriptstyle 8$};
	\draw (2) .. controls (2.5/4,.75/2) and (3.5/4,.75/2) ..  node [above=-2pt] { } (4); 
	\draw (1) .. controls (1.5/4,.75/2) and (2.5/4,.75/2) ..  node [above=-2pt] { } (3); 
	\draw (6) .. controls (6.5/4,.75/2) and (7.5/4,.75/2) ..  node [above=-2pt] { } (8); 
	\draw (5) .. controls (5.25/4,.5/2) and (5.75/4,.5/2) ..  node [above=-2pt] { } (6); 
	\draw  [densely dotted] (4) .. controls (4.25/4,.5/2) and (4.75/4,.5/2) ..  node [above=-2pt] { } (5); 
\end{tikzpicture} }
+ \kappa_{  \begin{tikzpicture}[baseline=.2cm]
	\foreach \x in {1,2,3,4,5,6,7,8} 
		\node (\x) at (\x/4,0) [inner sep=-1pt] {$\scriptstyle \bullet$};
	\node at (1/4,-.2) {$\scriptscriptstyle 2$};
	\node at (2/4,-.2) {$\scriptscriptstyle 4$};
	\node at (3/4,-.2) {$\scriptscriptstyle 5$};
	\node at (4/4,-.2) {$\scriptscriptstyle 6$};
	\node at (5/4,-.2) {$\scriptscriptstyle 1$};
	\node at (6/4,-.2) {$\scriptscriptstyle 3$};
	\node at (7/4,-.2) {$\scriptscriptstyle 7$};
	\node at (8/4,-.2) {$\scriptscriptstyle 8$};
	\draw (2) .. controls (2.5/4,.75/2) and (3.5/4,.75/2) ..  node [above=-2pt] { } (4); 
	\draw (1) .. controls (1.5/4,.75/2) and (2.5/4,.75/2) ..  node [above=-2pt] { } (3); 
	\draw (6) .. controls (6.5/4,.75/2) and (7.5/4,.75/2) ..  node [above=-2pt] { } (8); 
	\draw (5) .. controls (5.25/4,.5/2) and (5.75/4,.5/2) ..  node [above=-2pt] { } (6); 
	\draw  [densely dotted]  (4) .. controls (5/4,1/2) and (6/4,1/2) ..  node [above=-2pt] { } (7); 
\end{tikzpicture} }
+ \kappa_{  \begin{tikzpicture}[baseline=.2cm]
	\foreach \x in {1,2,3,4,5,6,7,8} 
		\node (\x) at (\x/4,0) [inner sep=-1pt] {$\scriptstyle \bullet$};
	\node at (1/4,-.2) {$\scriptscriptstyle 2$};
	\node at (2/4,-.2) {$\scriptscriptstyle 4$};
	\node at (3/4,-.2) {$\scriptscriptstyle 5$};
	\node at (4/4,-.2) {$\scriptscriptstyle 6$};
	\node at (5/4,-.2) {$\scriptscriptstyle 1$};
	\node at (6/4,-.2) {$\scriptscriptstyle 3$};
	\node at (7/4,-.2) {$\scriptscriptstyle 7$};
	\node at (8/4,-.2) {$\scriptscriptstyle 8$};
	\draw (2) .. controls (2.5/4,.75/2) and (3.5/4,.75/2) ..  node [above=-2pt] { } (4); 
	\draw (1) .. controls (1.5/4,.75/2) and (2.5/4,.75/2) ..  node [above=-2pt] { } (3); 
	\draw (6) .. controls (6.5/4,.75/2) and (7.5/4,.75/2) ..  node [above=-2pt] { } (8); 
	\draw (5) .. controls (5.25/4,.5/2) and (5.75/4,.5/2) ..  node [above=-2pt] { } (6); 
	\draw  [densely dotted]  (3) .. controls (3.5/4,.75/2) and (4.5/4,.75/2) ..  node [above=-2pt] { } (5); 
	\draw  [densely dotted]  (4) .. controls (5/4,1/2) and (6/4,1/2) ..  node [above=-2pt] { } (7); 
\end{tikzpicture} }
+ \kappa_{  \begin{tikzpicture}[baseline=.2cm]
	\foreach \x in {1,2,3,4,5,6,7,8} 
		\node (\x) at (\x/4,0) [inner sep=-1pt] {$\scriptstyle \bullet$};
	\node at (1/4,-.2) {$\scriptscriptstyle 2$};
	\node at (2/4,-.2) {$\scriptscriptstyle 4$};
	\node at (3/4,-.2) {$\scriptscriptstyle 5$};
	\node at (4/4,-.2) {$\scriptscriptstyle 6$};
	\node at (5/4,-.2) {$\scriptscriptstyle 1$};
	\node at (6/4,-.2) {$\scriptscriptstyle 3$};
	\node at (7/4,-.2) {$\scriptscriptstyle 7$};
	\node at (8/4,-.2) {$\scriptscriptstyle 8$};
	\draw (2) .. controls (2.5/4,.75/2) and (3.5/4,.75/2) ..  node [above=-2pt] { } (4); 
	\draw (1) .. controls (1.5/4,.75/2) and (2.5/4,.75/2) ..  node [above=-2pt] { } (3); 
	\draw (6) .. controls (6.5/4,.75/2) and (7.5/4,.75/2) ..  node [above=-2pt] { } (8); 
	\draw (5) .. controls (5.25/4,.5/2) and (5.75/4,.5/2) ..  node [above=-2pt] { } (6); 
	\draw  [densely dotted]  (3) .. controls (4/4,1/2) and (6/4,1/2) ..  node [above=-2pt] { } (7); 
	\draw  [densely dotted] (4) .. controls (4.25/4,.5/2) and (4.75/4,.5/2) ..  node [above=-2pt] { } (5); 
\end{tikzpicture} } .
 $$
However, to prove Theorem~\ref{thm:M}, we need to invert this equation, which is given by
\begin{equation} \label{kappaToproducts}
\kappa_{(\phi,\nu)}=\sum_{\mu\in \cS^\phi} \mathrm{mb}_{\preceq}(\mu,\nu)\kappa_{(\phi_1, \mu_1)}\cdots \kappa_{(\phi_\ell,\mu_\ell)},
\end{equation}
where $ \mathrm{mb}_{\preceq}(\mu,\nu)$ is the M\"obius function for $\preceq$ (computed in Proposition \ref{prop:atomic}, below) and $(\phi,\mu)=(\phi_1, \mu_1)\cdots (\phi_\ell,\mu_\ell)$ is the unique factorization into atomics.  

The order $\preceq$ is almost the same as the $\le_*$ in 
Proposition~10 of  \cite{BZ}. The difference between the two orders is that some particular covers in the order $\le_*$ are not present in our order $\preceq$.
Since $(\phi,\lambda) \prec (\phi,\mu)$ implies $\lambda\subseteq \mu$, we may write $\mu= \lambda \cup (\mu \setminus \lambda)$.

We say $(\phi,\lambda)\prec (\phi,\mu)$ is \emph{minimal} if for any subset $\mathcal{B}\subseteq (\mu \setminus \lambda)$, the set partition $\mathcal{B}\cup \lambda\prec\mu$.  For example, when $K=[4]$ and $\phi=1234$, then in the poset we have
$$
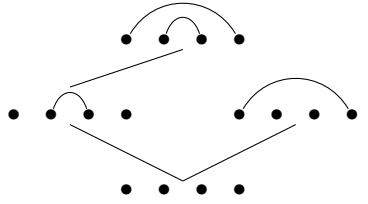
\begin{tikzpicture}[scale=.5,baseline=0]
	\foreach \x/\y in {0/0,1/0,2/0,3/0,3/-2,4/-2,5/-2,6/-2,3/2,4/2,5/2,6/2,6/0,7/0,8/0,9/0}
		\node (\x\y) at (\x,\y) [inner sep = 0pt] {$\bullet$};
	\draw (10) .. controls (1.25,.75) and (1.75,.75) .. (20);
	\draw (60) .. controls (6.75,1.25) and (8.25,1.25) .. (90);
	\draw (32) .. controls (3.75,3.25) and (5.25,3.25) .. (62);
	\draw (42) .. controls (4.25,2.75) and (4.75,2.75) .. (52);
	\draw (4.5,-1.75) -- (1.5,-.25);
	\draw (4.5,-1.75) -- (7.5,-.25);
	\draw (1.5,.75) -- (4.5,1.75);
\end{tikzpicture}\ ,
$$
so $\{1\larc{}4,2\larc{}3\}$ is not minimal over $\emptyset$ since $\emptyset\cup\{1\larc{}4\}\not\preceq\{1\larc{}4,2\larc{}3\}$.   Note that if $(\phi,\mu)\in A_\phi(\phi,\lambda)$, then $(\phi,\mu)\succeq(\phi,\lambda)$ is minimal.

The order $\preceq$ is a lower-semilattice in the sense that for any $(\phi,\nu)$ and $(\phi,\mu)$ there is a unique maximal element
 $(\phi,\nu)\wedge(\phi,\mu)$ such that  
 $(\phi,\nu)\wedge(\phi,\mu)\preceq (\phi,\nu)$ and $(\phi,\nu)\wedge(\phi,\mu)\preceq (\phi,\mu)$.

\begin{proposition}\label{prop:atomic} For $ (\phi,\lambda) \preceq (\phi,\nu)$ we have that the M\"obius function
  $$\mathrm{mb}_\preceq((\phi,\lambda),(\phi,\nu)) = \left\{ \begin{array}{ll} (-1)^{|\nu|-|\lambda|} & \hbox{if $(\phi,\nu)$ minimal,} \\ 
            &\\0&\hbox{otherwise.}  \end{array}\right.
 $$
 \end{proposition}

 \begin{proof} 
 If $ (\phi,\lambda) \preceq (\phi,\nu)$ is minimal, then the interval 
 $$[ (\phi,\lambda), (\phi,\nu)]_ \preceq =\{(\phi,\gamma) \mid  (\phi,\lambda) \preceq(\phi,\gamma) \preceq (\phi,\nu)\}$$
is exactly the poset on subsets ${\mathcal B}\subseteq (\nu \setminus \lambda)$. Hence the M\"obius function is the one given for this case. 

If $(\phi,\lambda) \preceq (\phi,\nu)$ is not minimal, then consider the set of elements that $\nu$ covers,
$$C_{\phi,\lambda,\nu}=\{\gamma \mid (\phi,\lambda) \preceq (\phi,\gamma) \prec\!\!\!\cdot\ (\phi,\nu)\}.$$ 
By inclusion-exclusion,
\begin{align*}
 \mathrm{mb}_\preceq((\phi,\lambda),(\phi,\nu)) &= - \sum _{ \lambda \preceq \gamma\prec \nu}  \mathrm{mb}_\preceq((\phi,\lambda),(\phi,\gamma))\\
&= - \sum_{\emptyset \neq N\subseteq C_{\phi,\lambda,\nu}} (-1)^{|N|}
                        \sum _{\lambda \preceq \rho \preceq \bigwedge_{\gamma\in N}\gamma}  \mathrm{mb}_\preceq((\phi,\lambda),(\phi,\rho)).
 \end{align*}
If we can show that for all sets $N\neq \emptyset$ that $\lambda\prec  \bigwedge_{\gamma\in N} \gamma$, then, as an alternating sum of M\"obius functions summed over nontrivial intervals,  $\mathrm{mb}_\preceq((\phi,\lambda),(\phi,\nu))=0$.  

Since $(\phi,\nu)$ is not minimal, there is $\emptyset\neq {\mathcal B}\subset (\nu \setminus \lambda)$ with $(\phi,\lambda)\prec (\phi,\lambda\cup \mathcal{B})\not\preceq(\phi,\nu)$.  Thus, there exists $i\larc{}l\in (\nu \setminus \lambda)\setminus \mathcal{B}$ such that $i\larc{}l$ does not connect atomics in $\lambda\cup \mathcal{B}$.  Fix such an $i\larc{}l$.  

Either $(\phi,\lambda\cup\{i\larc{}l\})\preceq (\phi,\nu)$ or $(\phi,\lambda\cup\{i\larc{}l\})\not\preceq (\phi,\nu)$.  In the first case,
$$(\phi,\lambda\cup\{i\larc{}l\})\preceq   \bigwedge_{\gamma\in C_{\phi,\lambda,\nu}}(\phi,\gamma),$$
since $\mathcal{B}\neq \emptyset$.   In the latter case, there exists $j\larc{}k\in(\nu \setminus \lambda)\setminus\{i\larc{}l\}$ with $i\prec_\phi j\prec_\phi k\prec_\phi l$.  Fix such a $j\larc{}k$ with $k-j$ minimal.  Then since $\lambda\cup((\nu \setminus \lambda)\setminus\{j\larc{}k\})\not\prec \nu$,
$$(\phi,\lambda\cup\{j\larc{}k\})\preceq   \bigwedge_{\gamma\in C_{\phi,\lambda,\nu}}(\phi,\gamma).\qedhere$$
 \end{proof}

 For example, let 
 $(\phi,\lambda)=
  \begin{tikzpicture}[baseline=.2cm]
	\foreach \x in {1,2,3,4,5,6,7,8} 
		\node (\x) at (\x/4,0) [inner sep=-1pt] {$\scriptstyle \bullet$};
	\node at (1/4,-.2) {$\scriptscriptstyle 2$};
	\node at (2/4,-.2) {$\scriptscriptstyle 4$};
	\node at (3/4,-.2) {$\scriptscriptstyle 5$};
	\node at (4/4,-.2) {$\scriptscriptstyle 6$};
	\node at (5/4,-.2) {$\scriptscriptstyle 1$};
	\node at (6/4,-.2) {$\scriptscriptstyle 3$};
	\node at (7/4,-.2) {$\scriptscriptstyle 7$};
	\node at (8/4,-.2) {$\scriptscriptstyle 8$};
	\draw (6) .. controls (6.5/4,.75/2) and (7.5/4,.75/2) ..  node [above=-2pt] { } (8); 
\end{tikzpicture} $ 
and 
$(\phi,\nu)=
  \begin{tikzpicture}[baseline=.2cm]
	\foreach \x in {1,2,3,4,5,6,7,8} 
		\node (\x) at (\x/4,0) [inner sep=-1pt] {$\scriptstyle \bullet$};
	\node at (1/4,-.2) {$\scriptscriptstyle 2$};
	\node at (2/4,-.2) {$\scriptscriptstyle 4$};
	\node at (3/4,-.2) {$\scriptscriptstyle 5$};
	\node at (4/4,-.2) {$\scriptscriptstyle 6$};
	\node at (5/4,-.2) {$\scriptscriptstyle 1$};
	\node at (6/4,-.2) {$\scriptscriptstyle 3$};
	\node at (7/4,-.2) {$\scriptscriptstyle 7$};
	\node at (8/4,-.2) {$\scriptscriptstyle 8$};
	\draw (6) .. controls (6.5/4,.75/2) and (7.5/4,.75/2) ..  node [above=-2pt] { } (8); 
	\draw [densely dotted]  (1) .. controls (1.5/4,.75/2) and (2.5/4,.75/2) ..  node [above=-2pt] { } (3); 
	\draw [densely dotted]  (2) .. controls (2.5/4,.75/2) and (3.5/4,.75/2) ..  node [above=-2pt] { } (4); 
	\draw  [densely dotted]  (3) .. controls (4/4,1/2) and (6/4,1/2) ..  node [above=-2pt] { } (7); 
	\draw  [densely dotted] (4) .. controls (4.25/4,.5/2) and (4.75/4,.5/2) ..  node [above=-2pt] { } (5); 
\end{tikzpicture}$. 
Then $\nu \setminus \lambda=\{(1,3),(2,4),(3,7),(4,5)\}$ and
$$C_{\phi,\lambda,\nu}=\Big\{
  \begin{tikzpicture}[baseline=.2cm]
	\foreach \x in {1,2,3,4,5,6,7,8} 
		\node (\x) at (\x/4,0) [inner sep=-1pt] {$\scriptstyle \bullet$};
	\node at (1/4,-.2) {$\scriptscriptstyle 2$};
	\node at (2/4,-.2) {$\scriptscriptstyle 4$};
	\node at (3/4,-.2) {$\scriptscriptstyle 5$};
	\node at (4/4,-.2) {$\scriptscriptstyle 6$};
	\node at (5/4,-.2) {$\scriptscriptstyle 1$};
	\node at (6/4,-.2) {$\scriptscriptstyle 3$};
	\node at (7/4,-.2) {$\scriptscriptstyle 7$};
	\node at (8/4,-.2) {$\scriptscriptstyle 8$};
	\draw (6) .. controls (6.5/4,.75/2) and (7.5/4,.75/2) ..  node [above=-2pt] { } (8); 
	\draw [densely dotted]  (2) .. controls (2.5/4,.75/2) and (3.5/4,.75/2) ..  node [above=-2pt] { } (4); 
	\draw  [densely dotted]  (3) .. controls (4/4,1/2) and (5/4,1/2) ..  node [above=-2pt] { } (6); 
	\draw  [densely dotted] (4) .. controls (4.25/4,.5/2) and (4.75/4,.5/2) ..  node [above=-2pt] { } (5); 
\end{tikzpicture}, 
  \begin{tikzpicture}[baseline=.2cm]
	\foreach \x in {1,2,3,4,5,6,7,8} 
		\node (\x) at (\x/4,0) [inner sep=-1pt] {$\scriptstyle \bullet$};
	\node at (1/4,-.2) {$\scriptscriptstyle 2$};
	\node at (2/4,-.2) {$\scriptscriptstyle 4$};
	\node at (3/4,-.2) {$\scriptscriptstyle 5$};
	\node at (4/4,-.2) {$\scriptscriptstyle 6$};
	\node at (5/4,-.2) {$\scriptscriptstyle 1$};
	\node at (6/4,-.2) {$\scriptscriptstyle 3$};
	\node at (7/4,-.2) {$\scriptscriptstyle 7$};
	\node at (8/4,-.2) {$\scriptscriptstyle 8$};
	\draw (6) .. controls (6.5/4,.75/2) and (7.5/4,.75/2) ..  node [above=-2pt] { } (8); 
	\draw [densely dotted]  (1) .. controls (1.5/4,.75/2) and (2.5/4,.75/2) ..  node [above=-2pt] { } (3); 
	\draw [densely dotted]  (2) .. controls (2.5/4,.75/2) and (3.5/4,.75/2) ..  node [above=-2pt] { } (4); 
	\draw  [densely dotted] (4) .. controls (4.25/4,.5/2) and (4.75/4,.5/2) ..  node [above=-2pt] { } (5); 
\end{tikzpicture} \Big\}.$$

\begin{proof}[Proof of Theorem~\ref{thm:M}] When $K=\emptyset$, the result is true by convention. For $|K|>0$, it suffices to verify that for all $(\phi,\lambda)$ the antipode $S$ satisfies the defining relation
 \begin{equation}\label{AntipodeDefiningRelation}
  0 = m(S\otimes \Id)\Delta (\kappa_{(\phi,\lambda)}) = m(\Id \otimes S)\Delta (\kappa_{(\phi,\lambda)}).
    \end{equation}

The coproduct on the $\kappa$ basis is 
 $$ \Delta (\kappa_{(\phi,\lambda)}) =  \sum_{ K=I\sqcup J \atop \lambda=\lambda_I\cup \lambda_J } \kappa_{(\phi|_I,\lambda_I)}  \otimes \kappa_{(\phi|_J,\lambda_J)},
  $$
Plugging this formula into the right-hand side of (\ref{AntipodeDefiningRelation}) with our proposed formula for the antipode, we obtain
 \begin{equation}\label{eq:antiM}
\sum_{K=I\sqcup J\atop \lambda=\lambda_I\cup \lambda_J} \sum_{\tau  \in L[J] \atop \lambda\in \cS^{\phi|_I\tau}}
                                                      \sum_{\nu\in A_{\phi|_J}(\tau,\lambda_J)}  
                                                         \hspace{-.5cm} (-1)^{|\nu|-|\lambda_J| + k }\kappa_{(\phi|_I,\lambda_I)} \kappa_{(\tau_1,\nu_1)}  \cdots \kappa_{(\tau_k,\nu_k)}.
 \end{equation}

To show (\ref{eq:antiM}) is zero, we will construct a sign reversing involution among the terms.
Let $\varphi=\phi|_I$ and let $\tilde{A}_{\varphi}(\varphi,\lambda_I)$ be the set of minimal elements above $(\varphi,\lambda_I)$ in the order $\preceq$.  By (\ref{kappaToproducts}) and Proposition \ref{prop:atomic},
  $$ \kappa_{(\varphi,\lambda_I)} = \sum_{(\varphi,\nu)\in \tilde{A}_{\varphi}(\varphi,\lambda_I)} (-1)^{|\nu|-|\lambda_I|} 
         \kappa_{(\varphi_1,\nu_1)}  \cdots \kappa_{(\varphi_r,\nu_r)}
  $$
 where $(\varphi,\nu)=(\varphi_1,\nu_1)  \cdots (\varphi_r,\nu_r)$ is the unique factorization into (minimal) atomics. Note that while  $\tau=\tau_1 \cdots \tau_k$ in (\ref{eq:antiM}) is the unique factorization into maximal increasing  subsequences with respect to $\phi|_J$, the decomposition $\varphi= \varphi_1 \cdots \varphi_r$ is a factorization into increasing subsequences with respect to $\phi|_I$ (without requiring maximality).   This distinction will produce the sign reversing involution. 
 
Since $|\lambda|=|\lambda_I|+|\lambda_J|,$ (\ref{eq:antiM}) is equal to 
\begin{equation}\label{eq:antiM2}
 \sum_{K=I\sqcup J ,\ \tau  \in L[J] \atop { \lambda\in \cS^{\phi|_I\tau}
                                             \atop { \nu\in A_{\phi|_J}(\tau,\lambda_J)
                                             \atop  \gamma\in \tilde{A}_{\varphi}(\varphi,\lambda_I) } }  }
          \hskip -12pt           (-1)^{|\nu|+|\gamma|-|\lambda| + k }\kappa_{(\varphi_1,\gamma_1)}  \cdots \kappa_{(\varphi_r,\gamma_r)} \kappa_{(\tau_1,\nu_1)}  \cdots \kappa_{(\tau_k,\nu_k)} .
 \end{equation}
 Define the involution $\iota$ on the set
 $$\left\{ \pm\kappa_{(\varphi_1,\gamma_1)}  \cdots \kappa_{(\varphi_r,\gamma_r)} \kappa_{(\tau_1,\nu_1)}  \cdots \kappa_{(\tau_k,\nu_k)}\  \bigg| \  \begin{array}{l} \scriptstyle K=I\sqcup J ,\, \tau  \in L[J],\lambda\in \cS^{\phi|_I\tau},\\ \scriptstyle \nu\in A_{\phi|_J}(\tau,\lambda_J),\\  \scriptstyle \gamma\in \tilde{A}_{\varphi}(\varphi,\lambda_I)\end{array}\right\}
 $$
 by
 \begin{align*}
 \iota( & \kappa_{(\varphi_1,\gamma_1)}  \cdots \kappa_{(\varphi_r,\gamma_r)} \kappa_{(\tau_1,\nu_1)}  \cdots \kappa_{(\tau_k,\nu_k)})\\
 &=\left\{\begin{array}{ll} - \big(\kappa_{(\varphi_1,\gamma_1)}  \cdots \kappa_{(\varphi_r,\gamma_r)} \kappa_{(\tau_1,\nu_1)}\big) \kappa_{(\tau_2,\nu_2)} \cdots \kappa_{(\tau_k,\nu_k)} & \text{if $\varphi_r\prec_\phi\tau_1$,}\\
 - \kappa_{(\varphi_1,\gamma_1)}  \cdots \kappa_{(\varphi_{r-1},\gamma_{r-1})}\big(\kappa_{(\varphi_r,\gamma_r)} \kappa_{(\tau_1,\nu_1)} \cdots \kappa_{(\tau_k,\nu_k)}\big) & \text{otherwise,}
 \end{array}\right.
 \end{align*}
where in each case we think of the $\tau$ sequence as either gaining or losing a term (therefore $k$ increases or decreases).

This is the desired sign reversing involution that shows that (\ref{eq:antiM2}) is zero.
 \end{proof}

\begin{remark} Note that the formula in Theorem~\ref{thm:M} is multiplicity free and cancelation free.

 With respect to the poset defined by $\lhd$, the formula in Theorem~\ref{thm:M} is upper-triangular.  In fact, if  $(\phi,\lambda)=(\phi_1,\lambda_1)\cdots(\phi_k,\lambda_k)$ the unique atomic factorization, the largest term in $S(\kappa_{(\phi,\lambda)})$ is $(-1)^k\kappa_{(\phi_k\cdots\phi_2\phi_1,\lambda)}$.

The image of the formula in Theorem~\ref{thm:M} under the functor $\overline{\mathcal K}$ is not as pleasing as the one we obtained in Corollary~\ref{cor:P}. We leave it to the reader.
\end{remark}


\section{Antipode on the supercharacter basis}\label{sec:Xbasis}

In this section, we first show that the antipode has a triangularity property on the $\chi$-basis, and then provide a formula for the antipode  in the special case where the underlying set partition has only one arc.  However, even this case leads to a surprising combinatorial identity.

\subsection{Triangularity of Antipode on the $\chi$ basis}\label{sec:Xbasis}

Define a partial order on $\bigcup_{\phi\in L[K]}\cS^\phi$ by $(\tau,\mu)\lhd_d (\phi,\lambda)$ if $\dim(\tau,\mu)<\dim(\phi,\lambda)$.   
Define constants $a_{I,J,\nu,\mu}^{\phi,\lambda} $ by
$$\Delta(\chi^{(\phi,\lambda)})=\sum_{K=I\sqcup J}\sum_{\nu\in \cS^{\phi|_I}\atop \mu\in \cS^{\phi|_{J}}}a_{I,J,\nu,\mu}^{\phi,\lambda} \chi^{(\phi|_I,\nu)}\otimes \chi^{(\phi|_{J},\mu)}.$$
Note that $a_{I,J,\nu,\mu}^{\phi,\lambda}\neq 0$ implies that $\dim(\phi|_I,\nu)+\dim(\phi|_{J},\mu)\leq \dim(\phi,\lambda)$ and   
$$\dim(\phi|_I,\nu)+\dim(\phi|_{J},\mu)= d(\phi,\lambda)\qquad  \text{only if}\qquad \lambda=\mu\cup \nu.$$

\begin{proposition} \label{AtomicXAntipode}
If $\lambda\in\cS^\phi$ is atomic, then 
$$S(\chi^{(\phi,\lambda)})=-\chi^{(\phi,\lambda)}-\sum_{\dim(\tau,\mu)<\dim(\phi,\lambda)}c_{\phi,\lambda}^{\tau,\mu}\chi^{(\tau,\mu)},$$
for some $c_{\phi,\lambda}^{\tau,\mu}\in \ZZ[q]$.
\end{proposition}

\begin{proof}
We induct on $|K|$.  If $|K|=0$, then the result is clear.  Suppose $|K|>0$.  Then by the recursive formula  for the antipode
$$S(\chi^{(\phi,\lambda)})=-\chi^{(\phi,\lambda)}-\sum_{(I,J)\models K}\sum_{\nu\in \cS^{\phi|_I}\atop \mu\in \cS^{\phi|_{J}}}a_{I,J,\nu,\mu}^{\phi,\lambda}S(\chi^{(\phi|_I,\nu)}) \chi^{(\phi|_{J},\mu)}.$$
Since $(\phi,\lambda)$ is atomic, $\nu\cup \mu\neq \lambda$ for all summands.  Thus, if $a_{I,J,\nu,\mu}^{\phi,\lambda}\neq 0$, 
$$\dim(\phi|_I,\nu)+\dim(\phi|_{J},\mu)<\dim(\phi,\lambda).$$
Furthermore, since $\chi^{(\phi',\mu')}\chi^{(\tau',\nu')}=\chi^{(\phi'\tau',\mu'\cup\nu')}$ and 
$$\dim(\phi'\tau',\mu'\cup\nu')=\dim(\phi',\mu')+\dim(\tau',\nu'),$$
the result follows by induction.
\end{proof}

 Note that  since $S$ is an anti-automorphism (see \cite{AM}), if $(\phi,\lambda)=(\phi_1,\lambda_1)\cdots(\phi_k,\lambda_k)$ is the unique factorization into atomics, then
$$S(\chi^{(\phi,\lambda)})=S(\chi^{(\phi_k,\lambda_k)})\cdots S(\chi^{(\phi_2,\lambda_2)})S(\chi^{(\phi_1,\lambda_1)}).$$
If we combine this observation with Proposition \ref{AtomicXAntipode}, we obtain the following corollary.

\begin{corollary}
Let $\lambda\in \cS^\phi$ with  $(\phi,\lambda)=(\phi_1,\lambda_1)\cdots(\phi_k,\lambda_k)$ the unique factorization into atomics.  Then
$$S(\chi^{(\phi,\lambda)})=(-1)^k\chi^{(\phi_k\cdots\phi_2\phi_1,\lambda)}+\sum_{\dim(\tau,\mu)<\dim(\phi,\lambda)} c_{\phi,\lambda}^{\tau,\mu}\chi^{(\tau,\mu)},$$
for some $c_{\phi,\lambda}^{\tau,\mu}\in \ZZ[q]$.
\end{corollary}

In particular, we obtain that the decomposition of the antipode on the $\chi$-basis is upper-triangular (up to reordering of the columns).  In fact, we can choose our basis so that the transition matrix is block upper-triangular with signed identity matrices as diagonal blocks. This result behaves well with the $\overline{\mathcal K}$ functor.

\begin{corollary} In the Hopf algebra $\overline{\mathcal K}({{\bf scf}(U)})$, for $n\ge 0$ and  $\lambda\in \cS^{\UIO_n}$, let $\lambda=\lambda_1|\cdots|\lambda_k$ be the unique factorization into atomics. We have
$$S(\chi^\lambda)=(-1)^k\chi^{\lambda_k|\cdots|\lambda_1}+\sum_{\dim(\mu)<\dim(\lambda)} \tilde{c}_{\lambda}^{\mu}\chi^{\mu},$$
where $\dim(\mu)=\dim(\UIO_n,\mu)$.
\end{corollary}

This behavior of the antipode on supercharacters was observed at the AIM workshop related to~\cite{BDTT10}. It is hard to envision a proof of this fact without the Hopf monoid techniques.

\subsection{A remarkable factorization formula}  Given an integer $n$, a {\sl skew ribbon shape} $\Gamma$ of size $n$ is a sequence of $n$ cells in the plane  where any two consecutive cells are such that the second cell is immediately east, south or southeast of the first. Given a finite set $K$ of size $n$ and a fixed $\tau\in L[K]$,
a \emph{standard filling} of $\Gamma$ is an injective map $\gamma\colon \Gamma\to K$ such that the values of $\gamma$ increase (with respect to $\tau$) in the rows of $\Gamma$ (from left to right) and in the columns of $\Gamma$ (from bottom to top).
For example
$$ \Gamma=\quad
\begin{tikzpicture}[scale=0.3,baseline=.5cm]
\foreach \x/\y/\num in {0/4/1,1/4/3,2/3/5,3/3/6, 3/2/2,4/1/9,4/0/4,5/0/8, 6/0/0}
{
  \draw (\x,\y) +(-0.5,-0.5) rectangle ++(0.5,0.5);
  \draw (\x,\y) node {};
}
\end{tikzpicture},\qquad
\gamma=\quad
\begin{tikzpicture}[scale=0.3,baseline=.5cm]
\foreach \x/\y/\num in {0/4/1,1/4/3,2/3/5,3/3/6, 3/2/2,4/1/9,4/0/4,5/0/7, 6/0/8}
{
  \draw (\x,\y) +(-0.5,-0.5) rectangle ++(0.5,0.5);
  \draw (\x,\y) node {$\scriptstyle \num$};
}
\end{tikzpicture}.
$$ 
Here, $\Gamma$ is a skew ribbon shape of size 9 and $\gamma$ is a standard filling of $\Gamma$ with $K=[9]$ and $\tau=\UIO_9$ the standard order on $K$. We denote by $\gamma_i$ the number of cells in the $i$th row of $\Gamma$ from top to bottom and $\gamma(r_i)$ the set of values of $\gamma$ in the $i$th row $r_i$ of $\Gamma$. In our example above, $\gamma_1=2$, $\gamma(r_1)=\{1,3\}$,  $\gamma_5=3$ and $\gamma(r_5)=\{4,7,8\}$. We denote by $\ell(\gamma)$ the number of rows of $\Gamma$. In the example $\ell(\gamma)=5$.

Fix a set $K$ and for $\tau\in L[K]$ let
\begin{equation}\label{MinMaxElements}
1_\tau= \mathrm{min}_\tau (K),\qquad \text{and}\qquad m_\tau=\mathrm{max}_\tau(K),
\end{equation}
be the minimum and maximum elements in $K$ with respect to $\tau$, respectively. Given $\phi\in L[K]$ and a skew ribbon shape $\Gamma$, we say that $\phi$ {\sl fits} $\Gamma$ {\sl with respect to $\tau$} if the filling $\gamma_\phi\colon \Gamma\to K$ obtained by assigning in increasing order (according to $\phi$) the values of $K$ to the cells of $\Gamma$ top to bottom, from left to right satisfies
\begin{enumerate}
\item[(i)] $\gamma_\phi$ is a standard filling with respect to $\tau$,
\item[(ii)] If the $j+1$st cell of $\Gamma$ is strictly southeast of the $j$th cell, then with respect to $\tau$ the $j$th element of $\phi$ is less than the $j+1$st element of $\phi$.
\end{enumerate}

In our running example above with $K=[9]$ and $\tau=\UIO_9$, $\phi=(1,3,5,6,2,9,4,7,8)$ fits the shape $\Gamma$ with respect to $\tau$ and $\gamma_\phi=\gamma$. 
In what follows, we are interested in ribbon tableaux $\gamma=\gamma_\phi$ where $1_\tau$ and $m_\tau$ are in distinct rows.
We denote by 
$$ \Phi_\tau(\phi) =\left\{ \gamma_\phi\colon\Gamma\to K\  \Big| \ { \phi \hbox{ fits } \Gamma \hbox{ with respect to } \tau, \hbox{ and for all } \atop 1\le i\le\ell(\gamma_\phi),  \big|\{1_\tau,m_\tau\}\cap\gamma_\phi(r_i)\big|\le 1     \hfill  }     \right\}.$$

Let $\phi=\phi_1\cdots\phi_k$ be the unique factorization into maximal rising subsequences with respect to $\tau$ and let $\ell_i$ denote the length of $\phi_i$.
We define statistics
\begin{align*}
\mathrm{mt}_\tau(\phi) &= \sum_{1\le i\le k;\ \ell_i>1}  \ell_i-2\\
\mathrm{io}_\tau(\phi) & = \big|\{1\leq i\leq k \mid \ell_i=1 \hbox{ and }1_\tau,m_\tau\notin\phi_i\}\big|\\
\mathrm{rst}_\tau(\phi) &= n-2-\mathrm{mt}_\tau(\phi)-\mathrm{io}_\tau(\phi).
\end{align*}
For $\phi=(1,3,5,6,2,9,4,8,7)$ and $\tau=\UIO_9$ we have $\mathrm{mt}(\phi)=2$, $\mathrm{io}(\gamma)=1$, and $\mathrm{rst}(\gamma)=4$.
\begin{theorem} \label{FactorizationTheorem} For any $K$ and $\phi,\tau\in L[K]$,
$$(t-1)^{\mathrm{mt}_\tau(\phi)}t^{\mathrm{rst}_\tau(\phi)}(t+1)^{\mathrm{io}_\tau(\phi)}=\sum_{ \gamma\in\Phi_\tau(\phi)} (-1)^{n-\ell(\gamma)} \prod_{1\leq i\leq \ell(\gamma)\atop 1_\tau,m_\tau\notin \gamma(r_i)} (\gamma_i t+1),$$
\end{theorem}

For example, this theorem says that if $\phi=\phi_1\phi_2=(1,3,5,6)(2,4)$ with respect to $\tau=\UIO_6$ and $K=[6]$, then
\begin{align*}
(t-1)^2t^2(t+1)^0 &  =-\overset{\begin{tikzpicture}[scale=0.2,baseline=.25cm]
\foreach \x/\y/\num in {0/2/1,1/2/3,2/2/5,3/1/6, 3/0/2,4/0/4}
{
  \draw (\x,\y) +(-0.5,-0.5) rectangle ++(0.5,0.5);
  \draw (\x,\y) node {$\scscs\num$};
}
\end{tikzpicture}}{(2t+1)}+\overset{\begin{tikzpicture}[scale=0.2,baseline=.25cm]
\foreach \x/\y/\num in {0/3/1,1/3/3,2/3/5,3/2/6, 3/1/2,4/0/4}
{
  \draw (\x,\y) +(-0.5,-0.5) rectangle ++(0.5,0.5);
  \draw (\x,\y) node {$\scscs\num$};
}
\end{tikzpicture}}{(t+1)^2}-\overset{\begin{tikzpicture}[scale=0.2,baseline=.25cm]
\foreach \x/\y/\num in {0/2/1,1/1/3,2/1/5,3/1/6, 3/0/2,4/0/4}
{
  \draw (\x,\y) +(-0.5,-0.5) rectangle ++(0.5,0.5);
  \draw (\x,\y) node {$\scscs\num$};
}
\end{tikzpicture}}{(2t+1)}+\overset{\begin{tikzpicture}[scale=0.2,baseline=.25cm]
\foreach \x/\y/\num in {0/3/1,1/2/3,2/2/5,3/2/6, 3/1/2,4/0/4}
{
  \draw (\x,\y) +(-0.5,-0.5) rectangle ++(0.5,0.5);
  \draw (\x,\y) node {$\scscs\num$};
}
\end{tikzpicture}}{(t+1)^2}-\overset{\begin{tikzpicture}[scale=0.2,baseline=.25cm]
\foreach \x/\y/\num in {0/2/1,1/2/3,2/1/5,3/1/6, 3/0/2,4/0/4} 
{
  \draw (\x,\y) +(-0.5,-0.5) rectangle ++(0.5,0.5);
  \draw (\x,\y) node {$\scscs\num$};
}
\end{tikzpicture}}{(2t+1)}
+\overset{\begin{tikzpicture}[scale=0.2,baseline=.25cm]
\foreach \x/\y/\num in {0/3/1,1/3/3,2/2/5,3/2/6, 3/1/2,4/0/4}
{
  \draw (\x,\y) +(-0.5,-0.5) rectangle ++(0.5,0.5);
  \draw (\x,\y) node {$\scscs\num$};
}
\end{tikzpicture}}{(t+1)^2}\\ 
&+\overset{\begin{tikzpicture}[scale=0.2,baseline=.25cm]
\foreach \x/\y/\num in {0/3/1,1/3/3,2/2/5,3/1/6, 3/0/2,4/0/4}
{
  \draw (\x,\y) +(-0.5,-0.5) rectangle ++(0.5,0.5);
  \draw (\x,\y) node {$\scscs\num$};
}
\end{tikzpicture}}{(2t+1)(t+1)}-\overset{\begin{tikzpicture}[scale=0.2,baseline=.25cm]
\foreach \x/\y/\num in {0/4/1,1/4/3,2/3/5,3/2/6, 3/1/2,4/0/4}
{
  \draw (\x,\y) +(-0.5,-0.5) rectangle ++(0.5,0.5);
  \draw (\x,\y) node {$\scscs\num$};
}
\end{tikzpicture}}{(t+1)^3}+\overset{\begin{tikzpicture}[scale=0.2,baseline=.25cm]
\foreach \x/\y/\num in {0/3/1,1/2/3,2/2/5,3/1/6, 3/0/2,4/0/4}
{
  \draw (\x,\y) +(-0.5,-0.5) rectangle ++(0.5,0.5);
  \draw (\x,\y) node {$\scscs\num$};
}
\end{tikzpicture}}{(2t+1)^2}-\overset{\begin{tikzpicture}[scale=0.2,baseline=.25cm]
\foreach \x/\y/\num in {0/4/1,1/3/3,2/3/5,3/2/6, 3/1/2,4/0/4}
{
  \draw (\x,\y) +(-0.5,-0.5) rectangle ++(0.5,0.5);
  \draw (\x,\y) node {$\scscs\num$};
}
\end{tikzpicture}}{(2t+1)(t+1)^2}
\\ 
&+\overset{\begin{tikzpicture}[scale=0.2,baseline=.25cm]
\foreach \x/\y/\num in {0/3/1,1/2/3,2/1/5,3/1/6, 3/0/2,4/0/4}
{
  \draw (\x,\y) +(-0.5,-0.5) rectangle ++(0.5,0.5);
  \draw (\x,\y) node {$\scscs\num$};
}
\end{tikzpicture}}{(2t+1)(t+1)}-\overset{\begin{tikzpicture}[scale=0.2,baseline=.25cm]
\foreach \x/\y/\num in {0/4/1,1/3/3,2/2/5,3/2/6, 3/1/2,4/0/4}
{
  \draw (\x,\y) +(-0.5,-0.5) rectangle ++(0.5,0.5);
  \draw (\x,\y) node {$\scscs\num$};
}
\end{tikzpicture}}{(t+1)^3}-\overset{\begin{tikzpicture}[scale=0.2,baseline=.25cm]
\foreach \x/\y/\num in {0/4/1,1/3/3,2/2/5,3/1/6, 3/0/2,4/0/4}
{
  \draw (\x,\y) +(-0.5,-0.5) rectangle ++(0.5,0.5);
  \draw (\x,\y) node {$\scscs\num$};
}
\end{tikzpicture}}{(2t+1)(t+1)^2}+\overset{\begin{tikzpicture}[scale=0.2,baseline=.25cm]
\foreach \x/\y/\num in {0/5/1,1/4/3,2/3/5,3/2/6, 3/1/2,4/0/4}
{
  \draw (\x,\y) +(-0.5,-0.5) rectangle ++(0.5,0.5);
  \draw (\x,\y) node {$\scscs\num$};
}
\end{tikzpicture}}{(t+1)^4}. 
\end{align*}

This theorem does not really depend on our particular choice of $K$ and $\tau$, so it is sufficient to get the result for $K=[n]$ and $\tau=\UIO_n$. We will assume this for the proof below.
Before proving the general theorem we require a few lemmas.  The proof of the first lemma was point out to us by E. Kruskal.  This case addresses the special case when $\phi=\tau$. In this case, $\phi$ fits any skew ribbon shape $\Gamma$ that is not a single row and with no two cells in the same column. Moreover, the set $\{(\gamma_1,\ldots,\gamma_\ell) \mid \gamma\in\Phi_\phi(\phi)\}$  is exactly the set of integer composition $\gamma=(\gamma_1,\ldots,\gamma_\ell)\vDash n$ such that $\ell=\ell(\gamma)\geq 2$.

\begin{lemma} \label{(n)Lemma} For $n\geq 2$ and $\phi=\tau=\UIO_n$,
$$(t-1)^{n-2}=\sum_{\gamma\vDash n;\ \ell(\gamma)\geq 2} (-1)^{n-\ell}(\gamma_2t+1)\cdots (\gamma_{\ell-1}t+1).$$
\end{lemma}
\begin{proof}
We induct on $n$.  When $n=2$, we have that $1=(-1)^{2-2}\cdot 1$, as desired.  

Suppose $n>2$.  Then
\begin{align*}
\sum_{\gamma\vDash n;\  \ell(\gamma)\geq 2}& (-1)^{n-\ell}(\gamma_2t+1)\cdots (\gamma_{\ell-1}t+1)\\
& = (-1)^{n-2}(n-1)\quad+
\sum_{\gamma\vDash n;\ \ell(\gamma)> 2} (-1)^{n-\ell}(\gamma_2t+1)\cdots (\gamma_{\ell-1}t+1).
\end{align*}
We separate the second sum into two cases and use induction hypothesis on the first one to obtain
\begin{align*}
&\sum_{\gamma\vDash n;\ \ell(\gamma)> 2} (-1)^{n-\ell}(\gamma_2t+1)\cdots (\gamma_{\ell-1}t+1)\\
&=\hspace{-.25cm}\sum_{{\gamma\vDash n;\ \ell(\gamma)> 2}\atop \gamma_2=1} \hspace{-.25cm}(-1)^{n-\ell}(\gamma_2t+1)\cdots (\gamma_{\ell-1}t+1)+\hspace{-.35cm}\sum_{{\gamma\vDash n;\ \ell(\gamma)> 2}\atop \gamma_2>1}\hspace{-.25cm} (-1)^{n-\ell}(\gamma_2t+1)\cdots (\gamma_{\ell-1}t+1).\\
&=\quad (t+1)(t-1)^{n-3}\quad + \quad \sum_{{\gamma\vDash n;\ \ell(\gamma)> 2}\atop \gamma_2>1}\hspace{-.25cm} (-1)^{n-\ell}(\gamma_2t+1)\cdots (\gamma_{\ell-1}t+1).
\end{align*}
For the second sum in this last equality, we let $\gamma_2=\mu_2+1$ and $\gamma_i=\mu_i$ for $i\ne 2$. 
\begin{align*}
&\sum_{{\gamma\vDash n;\ \ell(\gamma)> 2}\atop \gamma_2>1}(-1)^{n-\ell}(\gamma_2t+1)\cdots (\gamma_{\ell-1}t+1) \\
&=\hspace{-.25cm}\sum_{\mu\vDash n-1\atop \ell(\mu)> 2}(-1)^{n-\ell}(\mu_2t+1)\cdots (\mu_{\ell-1}t+1)+ t\sum_{\mu\vDash n-1\atop \ell(\mu)> 2} (-1)^{n-\ell}(\mu_3t+1)\cdots (\mu_{\ell-1}t+1).
\end{align*}
Focusing on the first of these two sums and using the induction hypothesis we have
\begin{align*}
\Big(\sum_{\mu\vDash n-1\atop \ell(\mu)> 2} &(-1)^{n-\ell}(\mu_2t+1)\cdots (\mu_{\ell-1}t+1)\Big) + (-1)^{n-2}(n-2)\\
&=-\sum_{\mu\vDash n-1\atop \ell(\mu)\geq 2} (-1)^{n-1-\ell}(\mu_2t+1)\cdots (\mu_{\ell-1}t+1)
\quad =\ -(t-1)^{n-3}.
\end{align*}
Substituting all this in the original equation, summing over $k=\mu_1$ where $\mu'_i=\mu_{i+1}$ and using the induction hypothesis again, we get
\begin{align*}
&\sum_{\gamma\vDash n;\ \ell(\gamma)\geq 2} (-1)^{n-\ell}(\gamma_2t+1)\cdots (\gamma_{\ell-1}t+1)\\
&=(-1)^{n-2}+t(t-1)^{n-3}+t\hspace{-.35cm}\sum_{\mu\vDash n-1;\ \ell(\mu)> 2}\hspace{-.25cm} (-1)^{n-\ell}(\mu_3t+1)\cdots (\mu_{\ell-1}t+1)\\
&=(-1)^{n-2}+t(t-1)^{n-3}+t\sum_{k=1}^{n-3}(-1)^k \hspace{-.25cm}\sum_{\mu'\vDash n-1-k\atop \ell(\mu')\geq 2} \hspace{-.35cm} (-1)^{n-k-\ell-1}(\mu'_2t+1)\cdots (\mu'_{\ell-1}t+1)\\
&=(-1)^{n-2}+t(t-1)^{n-3}+t\sum_{k=1}^{n-3}(-1)^k(t-1)^{n-k-3}\\
&=(-1)^{n-2}+t(t-1)^{n-3}+(-1)^{n-2}\big((1-t)^{n-3}-1\big) \quad =\quad (1-t)^{n-2}\,.
\end{align*}
\end{proof}

The following lemma addresses two additional special cases.
 
\begin{lemma}\label{(n,1)Lemma} Let $K=[n]$ and $\tau=\UIO_n$ \hfill

\begin{enumerate}
\item For $\phi=(n,1,2,\ldots, n-1)$ or $\phi=(2,\ldots, n,1)$, 
$$(t-1)^{n-3}t=\sum_{\gamma\in\Phi_\tau(\phi)} (-1)^{n-\ell(\gamma)} \prod_{1,n\notin \gamma(r_i)}(\gamma_it+1).$$
\item  For $\phi=(n,2,\ldots, n-1, 1)$,
$$(t-1)^{n-4}t^2=\sum_{\gamma\in\Phi_\tau(\phi)} (-1)^{n-\ell(\gamma)}\!\! \prod_{1,n\notin \gamma(r_i)}(\gamma_it+1).$$
\end{enumerate}
\end{lemma}
\begin{proof}
(1) Suppose that $\phi=(n,1,2,\ldots, n-1)$ (the case $\phi=(2,\ldots, n,1)$ is similar).  If $\gamma\in\Phi_\tau(\phi)$, then we have that $\gamma(r_1)=\{n\}$ and $1\in \gamma(r_2)$. The set $\Phi_\tau(\phi)$ is in bijection with the composition $\mu\vDash n-1$ where 
$\mu_i=\gamma_{i+1}$. In turn, we can add a part of size one at the end of $\mu$ and obtain
\begin{align*}
\sum_{\gamma\in\Phi_\tau(\phi)}& (-1)^{n-\ell(\gamma)}\!\! \prod_{1,n\notin \gamma(r_i)}(\gamma_it+1) \\
&=\sum_{\mu\vDash n-1} (-1)^{n-\ell(\mu)-1} (\mu_2t+1)\cdots (\mu_{\ell(\mu)}t+1)\\
&=\sum_{\mu\vDash n;\ \ell(\mu)\geq 2} (-1)^{n-\ell(\mu)} (\mu_2t+1)\cdots (\mu_{\ell(\mu)-1}t+1)\\
&\hspace*{2cm}-
\sum_{{\mu\vDash n;\ \ell(\mu)\geq 2}\atop \mu_\ell\neq 1} (-1)^{n-\ell(\mu)} (\mu_2t+1)\cdots (\mu_{\ell(\mu)-1}t+1)\\
&=\sum_{\mu\vDash n;\ \ell(\mu)\geq 2} (-1)^{n-\ell(\mu)} (\mu_2t+1)\cdots (\mu_{\ell(\mu)-1}t+1)\\
&\hspace*{2cm}+
\sum_{\mu\vDash n-1;\  \ell(\mu)\geq 2} (-1)^{n-1-\ell(\mu)} (\mu_2t+1)\cdots (\mu_{\ell(\gamma)-1}t+1)\\
&=\quad (t-1)^{n-2}+(t-1)^{n-3}\quad =\quad (t-1)^{n-3}t.
\end{align*}
where the last line follow from Lemma \ref{(n)Lemma}.

(2) For $\phi=(n,2,\ldots, n-1, 1)$, if $\gamma\in\Phi_\tau(\phi)$, then $\gamma(r_1)=\{n\}$ and $\gamma(r_{\ell(\gamma)})=\{1\}$.
Let $\phi'=(n,1,2,\ldots, n-1)$. The set $\{\gamma' \mid \gamma'\in\Phi_\tau({\phi'}),\ \gamma'_2=1\}$ is in bijection with the set $\Phi_\tau(\phi)$. 
Let $\phi''=(n-1,1,\ldots, n-2)$ and $\tau''=(1,\ldots,n-1)$.
Removing $1$ and reducing all other entries by $1$ for the $\gamma'$ in $\{\gamma' \mid \gamma'\in\Phi_\tau({\phi'}),\ \gamma'_2>1\}$ gives a bijection with the set $\Phi_{\tau''}(\phi'')$. We can then use (1) above to get:
\begin{align*}
&\sum_{\gamma\in\Phi_\tau(\phi)} (-1)^{n-\ell(\gamma)} \prod_{1,n\notin \gamma(r_i)}(\gamma_it+1) \\
&\quad =\hspace{-.15cm} \sum_{\gamma'\in\Phi_\tau({\phi'})} \hspace{-.15cm} (-1)^{n-\ell(\gamma')} \prod_{1,n\notin \gamma'(r_i)}(\gamma'_it+1)\ -\hspace{-.15cm}\sum_{\gamma'\in\Phi_\tau({\phi'})\atop \gamma'_2\neq 1}\hspace{-.15cm} (-1)^{n-\ell(\gamma')} \prod_{1,n\notin \gamma'(i)}(\gamma'_it+1)\\
&\quad = \qquad (t-1)^{n-3}t\qquad +\sum_{\gamma\in \Phi_{\tau''}(\phi'')} (-1)^{n-1-\ell(\gamma)} \prod_{1,n\notin \gamma(r_i)}(\gamma_it+1)\\
&\quad=\qquad  (t-1)^{n-3}t+(t-1)^{n-4}t\qquad=\qquad (t-1)^{n-4}t^2\,.
\end{align*}
\end{proof}

\begin{proof}[Proof of Theorem \ref{FactorizationTheorem}]
 We factor the expression 
$$z'_{\phi,\tau}(t)=\sum_{\gamma\in\Phi_\tau(\phi)} (-1)^{n-\ell(\gamma)} \prod_{1,n\notin \gamma(r_i)} (\gamma_i t+1)$$
as follow. Let $\phi=\phi_1\phi_2\cdots \phi_k$ be the unique factorization into maximal rising subsequences with respect to $\tau$. 
Define
$$(\hat{\phi}_i,\hat{\tau}_i)=\left\{\begin{array}{ll}
 \big((m_\tau,\phi_i,1_\tau),(1_\tau,\phi_i,m_\tau)\big)
 & \text{if $1_\tau,m_\tau\notin \phi_i$,}\\
\big((m_\tau,\phi_i),(\phi_i,m_\tau)\big)
& \text{if $1_\tau\in \phi_i, m_\tau\notin \phi_i$,}\\
 \big((\phi_i,1_\tau),(1_\tau,\phi_i)\big)
& \text{if $m_\tau\in \phi_i, 1_\tau\notin \phi_i$,}\\
(\phi_i,\phi_i)
& \text{if $1_\tau,m_\tau\in \phi_i$.}
\end{array}\right.$$
Then 
$$z'_{\phi,\tau}(t)=\prod_{i=1}^{k} z'_{\hat{\phi}_i,\hat{\tau}_i}(t),$$
and the result follow from Lemma \ref{(n)Lemma} and Lemma \ref{(n,1)Lemma}.
\end{proof}

\subsection{The antipode for the $\chi$-basis}

The following argument gives a general formula for the antipode on $\chi^{(\tau,1\slarc{}n)}$. 
Without loss of generality, we may assume that $K=[n]$ and $\tau=\UIO_n$.
Let $t=q-1$.
By Takeuchi's formula (\ref{TakeuchisFormula}) and Proposition \ref{RestrictionProposition} (1) and (4),
$$S(\chi^{(\tau,1\slarc{}n)})=\hspace{-.3cm}\sum_{(J_1,J_2,\ldots, J_\ell)\models [n]} \frac{(-1)^{\ell}}{q^{(\ell-1)(n-2)}t^{\ell-1}}\Res^{U_n}_{U_{J_1}}(\chi^{(\tau,1\slarc{a}n)})\cdots \Res_{U_{J_\ell}}^{U_n}(\chi^{(\tau,1\slarc{a}n)}).$$
By Proposition \ref{RestrictionProposition} (3),
$$\Res^{U_n}_{U_{J}}(\chi^{(\tau,1\slarc{}n)})=\left\{\begin{array}{ll} q^{|[n]\setminus J|} \chi^{(\tau|_J,1\slarc{}n)}, & \text{if $1,n\in J$,}\\
\displaystyle q^{|[n]\setminus J|} t\big(\chi^\emptyset +\sum_{j\in J\setminus\{n\}} \chi^{(\tau|_J,j\slarc{}n)}\big), & \text{if $1\notin J$, $n\in J$},\\
\displaystyle q^{|[n]\setminus J|} t\big(\chi^\emptyset +\sum_{j\in J\setminus\{1\} } \chi^{(\tau|_J,1\slarc{}j)}\big), & \text{if $1\in J$, $n\notin J$},\\
\displaystyle q^{|[n]\setminus J|} t\big((|J|t+1)\chi^\emptyset +t\!\!\!\!\! \sum_{i,j\in J, i<j} \!\!\!\!\! \chi^{(\tau|_J,i\slarc{}j)}\big), & \text{if $1,n\notin J$}.\end{array}\right.$$
By inspection, $\mathrm{Coeff}(S(\chi^{\tau,1\slarc{}n}), \chi^{(\phi,\mu)})\neq 0$ only if $\phi$ has a factorization
\begin{equation}\label{OrderDecomposition}
\phi = \phi_0(a_1\cdot \psi_1 \cdot b_1) \phi_1(a_2 \cdot \psi_2 \cdot b_2)\phi_2\  \cdots\  \phi_{s-1} (a_s \cdot\psi_s \cdot b_s) \phi_s,
\end{equation}
satisfying
\begin{enumerate}
 \item $\mu=\{a_i\larc{} b_i\mid 1\leq i\leq s, a_i\neq b_i\}$,
 \item the subsequence $a_i\psi_ib_i$ is increasing with respect to $\tau$.
\end{enumerate}
For purposes of counting (this has no effect on whether the coefficient is nonzero), let us add the condition
\begin{enumerate}
\item[(3)] $1\in \{a_1,a_2,\ldots, a_s\}$ and $n\in \{b_1,b_2,\ldots, b_s\}$, where we permit  $a_i=b_i$ if $a_i\in \{1,n\}$.
\end{enumerate}
Note that the last condition forces $1$ and $n$ to be included in the $a$'s and $b$'s, but if $\mu$ does not have an arc with an endpoint at $1$ or $n$, then we are permitted to let it disappear by letting $a_i=b_i$.

 For example, if $\mu=\{2\larc{}5,7\larc{}8, 13\larc{}16\}$, $\tau=\UIO_{16}$, and
 \begin{equation*}
  \phi=(11,12,10,7,8,9,1,6,13,14,16,2,3,4,5,15),
 \end{equation*}
 then the factorization for $\phi$ would be
 $$\phi=
 \begin{tikzpicture}[scale=.5,baseline=-.5cm]
 	\foreach \x/\num in {1/11,2/12,3/10,4/7,5/8,6/9,7/1,8/6,9/13,10/14,11/16,12/2,13/3,14/4,15/5,16/15}
		\node (\x) at (\x,0) {$\num$};
	\node at (2,-1) {$\phi_0$};
	\node at (4,-1) {$a_1$};
	\node at (5,-1) {$b_1$};
	\node at (6,-1) {$\phi_1$};
	\node at (7,-1) {$a_2$};
	\node at (7,-1.5) {$b_2$};
	\node at (8,-1) {$\phi_2$};
	\node at (9,-1) {$a_3$};
	\node at (10,-1) {$\psi_3$};
	\node at (11,-1) {$b_3$};
	\node at (12,-1) {$a_4$};
	\node at (13.5,-1) {$\psi_4$};
	\node at (15,-1) {$b_4$};
	\node at (16,-1) {$\phi_4$};
	\draw (4) .. controls (4.25,.75) and (4.75,.75) .. (5);
	\draw (7) .. controls (6.5, 1.25) and (7.5,1.25) .. (7);
	\draw (9) .. controls (9.5,1.25) and (10.5,1.25) .. (11);
	\draw (12) .. controls (12.75,1.75) and (14.25,1.75) .. (15);
	\draw [|-|] (.7,-.5) -- (3.3,-.5); 
	\draw [|-|] (12.7,-.5) -- (14.3,-.5); 
 \end{tikzpicture},
 $$
 where $a_2=b_2=1$ in this case since $1\larc{}j\notin \mu$ for any $j$.
 
We augment $\phi$ by letting $\phi'=b_0\phi a_{s+1}$ where $b_0=n+1$ and $a_{s+1}=0$.  For each $0\leq i\leq s$ factor the sequence $b_i\phi_i a_i=\phi^{(i,1)}\phi^{(i,2)}\cdots \phi^{(i,l_i)}$ into maximal rising sequences with respect to $\tau$.
Let 
$$z_{\phi,\tau}^\mu(t)=(t-1)^{\mathrm{mt}_{\tau}^\mu(\phi)}t^{\mathrm{rst}_\tau^\mu(\phi)}(t+1)^{\mathrm{io}_\tau^\mu(\phi)},$$
where
\begin{align*}
\mathrm{mt}_{\tau}^\mu(\phi) &=\sum_{\ell(\phi^{(i,j)})>1} \ell(\phi^{(i,j)})-2, \\ 
\mathrm{io}_\tau^\mu(\phi)&=  \sum_{i=0}^{s}\#\{2\leq j\leq l_i-1\mid \ell(\phi^{(i,j)})=1\},\\
\mathrm{rst}_\tau^\mu(\phi) &= \bigg(\sum_{j=0}^{s}\ell(\phi_j)\bigg) -\mathrm{mt}_{\tau}^\mu(\phi)-\mathrm{io}_\tau^\mu(\phi).
\end{align*}

\begin{remark}
The notation here is motivated by the fact that 
$$z_{\phi,\tau}^\emptyset(t)=(-1)^nz'_{\phi,\tau}(t)=(-1)^n(t-1)^{\mathrm{mt}_\tau(\phi)}t^{\mathrm{rst}_\tau(\phi)}(t+1)^{\mathrm{io}_\tau(\phi)},$$
which we saw in Theorem \ref{FactorizationTheorem}.
\end{remark}

Then we have the following theorem.

\begin{theorem} For $\tau\in L[K]$,
$$S(\chi^{(\tau,{1_\tau}\!\!\raise2pt\hbox{$\slarc{}$}m_\tau)})=\sum_{\phi \in L[K]}\ \sum_{\mu\in \cS^\phi \cap\cS^\tau} (-1)^s(q-1)^{s-1} z_{\phi,\tau}^\mu(q-1) \chi^{(\phi,\mu)},$$
where the sum is over all $\mu$ such that 
$$\phi = \phi_0(a_1\cdot \psi_1 \cdot b_1) \phi_1(a_2 \cdot \psi_2 \cdot b_2)\phi_2\  \cdots\  \phi_{s-1} (a_s \cdot\psi_s \cdot b_s) \phi_s$$ factors as in (\ref{OrderDecomposition}).
\end{theorem}

\begin{proof}  Let $t=q-1$.
First, consider for which set compositions $J\models [n]$ 
\begin{equation}\label{RestrictionNonzero}
\mathrm{Coeff}\bigg( \frac{(-1)^{\ell}}{q^{(\ell-1)(n-2)}t^{\ell-1}}\Res^{U_n}_{U_{J_1}}(\chi^{(\tau,1\slarc{}n)})\cdots \Res_{U_{J_\ell}}^{U_n}(\chi^{(\tau,1\slarc{}n)})
,\chi^{(\phi,\mu)}\bigg)\neq 0.
\end{equation}
Let $\tau_J=\tau|_{J_1}\tau|_{J_2}\cdots \tau|_{J_\ell}$, and let $\Gamma_J$ be the unique skew ribbon tableau with row sizes given by $J$ such that $\tau_J$ fits $\Gamma_J$ with respect to $\tau$.  Let $\gamma_J:\Gamma\rightarrow K$ be the corresponding ribbon tableau.
By Proposition \ref{RestrictionProposition} (3), condition (\ref{RestrictionNonzero}) is true if and only if $(\tau, \phi, \mu, J)$ satisfy
\begin{enumerate}
\item[(a)]  $\phi=\tau_J$ has a factorization (\ref{OrderDecomposition}) with respect to $\mu$,
\item[(b)]  there is a sequence $1\leq k_1< k_2<\cdots < k_s\leq \ell(\gamma_J)$ such that for each $1\leq i\leq s$, we have $a_i,b_i\in \gamma_J(r_{k_i})$.
\end{enumerate}
We say that a quadruple $(\tau,\phi,\mu,J)$ satisfying (a) and (b) is \emph{compatible}.  In this case,
\begin{align*}
\mathrm{Coeff}\bigg( \frac{(-1)^{\ell}}{q^{(\ell-1)(n-2)}t^{\ell-1}}\Res^{U_n}_{U_{J_1}}(\chi^{(\tau,1\slarc{}n)})\cdots& \Res_{U_{J_\ell}}^{U_n}(\chi^{(\tau,1\slarc{}n)})
,\chi^{(\phi,\mu)}\bigg)\\
&= t^{s-1} \prod_{k\notin \{k_1,k_2,\ldots, k_s\}} ((\gamma_J)_kt+1),
\end{align*}
since 
$$
s-1=\left\{\begin{array}{ll} |\mu|+1 & \text{if $1\larc{}n\notin \mu$, }\\
|\mu| & \text{if $1\larc{}n\in\mu$}. \end{array}\right.
$$
Thus,
\begin{align}
\mathrm{Coeff}\big(S(\chi^{\tau,1\slarc{}n}),\chi^{(\phi,\mu)}\big) &=\sum_{ (\tau,\phi,\mu,J)\atop\text{compatible}}
t^{s-1} \prod_{k\notin \{k_1,k_2,\ldots, k_s\}} ((\gamma_J)_kt+1)\notag\\
&=t^{s-1}\sum_{ (\tau,\phi,\mu,J)\atop\text{compatible}}(-1)^{\ell(\gamma_J)} \prod_{k\notin \{k_1,k_2,\ldots, k_s\}} ((\gamma_J)_kt+1),\label{eq:GeneralCoef}
\end{align}
since $s$ only depends on $\mu$. 

Next, we write (\ref{eq:GeneralCoef}) as a product sums, where each sum will be as in Theorem \ref{FactorizationTheorem}.   Note that in (\ref{eq:GeneralCoef}) the part $J_k$ contributes to the product only when  $k\notin\{k_1,k_2,\ldots,k_s\}$.   Thus, (\ref{eq:GeneralCoef}) factors as
\begin{equation}\label{eq:SumToProduct}
(-1)^st^{s-1}\prod_{0\leq i\leq s} \sum_{{\gamma:\Gamma\rightarrow K_i \atop b_i\phi_ia_{i+1}\text{ fits } \Gamma} \atop \ell(\gamma)> 1} (-1)^{\ell(\gamma)} \prod_{j=2}^{\ell(\gamma)-1} (\gamma_j t +1),
\end{equation}
where $K_i$ is the set of elements in the subsequence $b_i\phi_ia_{i+1}$.

The inner sums are similar to the sums in Theorem \ref{FactorizationTheorem}, except that the distinguished elements are the smallest and largest elements of the order $b_i\phi_ia_i$ instead of $\tau|_{K_i}$.  However, the following algorithm takes an arbitrary $\phi,\tau\in L[K]$ and constructs new orders $\tilde\phi,\tilde\tau\in L[\tilde{K}]$ such that 
$$\sum_{{\gamma:\Gamma\rightarrow K_i \atop \phi\text{ fits } \Gamma} \atop |\{1_\phi,m_\phi\}\cap \gamma(r_k)|\leq 1} (-1)^{\ell(\gamma)} \prod_{1_\phi,m_\phi\notin \gamma(r_j)} (\gamma_j t +1)=z_{\tilde\phi,\tilde\tau}^\emptyset(t).$$
\begin{description}
\item[Step 1]  Let $\phi=\phi_1\phi_2\cdots\phi_r$ be the maximal decomposition into increasing subsequences with respect to $\tau$ with $\phi_1=1_\phi \phi_1'$ and $\phi_r=\phi_r'm_\phi$.
\item[Step 2] Define $\tilde K=K\setminus\{1_\phi,m_\phi\}\cup \{m,M\}$, where $m$ and $M$ are elements not already in $K$.
\item[Step 3]  Define
\begin{align*}
\tilde\tau&=m\tau|_{K\cap\tilde{K}} M\\
\tilde\phi&=(\phi_1'M)\phi_2\cdots\phi_{r-1}(m\phi_{r}').
\end{align*}
\end{description}
Note that we have given $\tilde\phi$ in terms of its maximal decomposition into rising subsequences with respect to $\tilde\tau$, and the lengths of the sequences are the same as those for $\phi$ with respect to $\tau$.  Thus,
\begin{align*}
\sum_{{\gamma:\Gamma\rightarrow K_i \atop \phi\text{ fits } \Gamma} \atop \ell(\gamma)> 1} (-1)^{\ell(\gamma)} \prod_{j=2}^{\ell(\gamma)-1} (\gamma_j t +1) &= \hspace{-.5cm}\sum_{{\gamma:\Gamma\rightarrow K_i \atop \tilde\phi\text{ fits } \Gamma} \atop |\{0,M\}\cap \gamma(r_k)|\leq 1} (-1)^{\ell(\gamma)} \prod_{0,M\notin \gamma(r_j)} (\gamma_j t +1)\\
&=z_{\tilde\phi,\tilde\tau}^\emptyset(t),
\end{align*}
by Theorem \ref{FactorizationTheorem}.  

The theorem now follows by plugging the $z_{\tilde\phi,\tilde\tau}^\emptyset(t)$ into (\ref{eq:SumToProduct}).
\end{proof}

For example, let $\tau=\UIO_{16}$, $\mu=\{2\larc{} 5,7\larc{}8\}$, and 
 $$\phi=
 \begin{tikzpicture}[scale=.5,baseline=-.5cm]
 	\foreach \x/\num in {1/11,2/12,3/10,4/7,5/8,6/9,7/1,8/6,9/13,10/14,11/16,12/2,13/3,14/4,15/5,16/15}
		\node (\x) at (\x,0) {$\num$};
	\node at (2,-1) {$\phi_0$};
	\node at (4,-1) {$a_1$};
	\node at (5,-1) {$b_1$};
	\node at (6,-1) {$\phi_1$};
	\node at (7,-1) {$a_2$};
	\node at (7,-1.5) {$b_2$};
	\node at (9,-1) {$\phi_2$};
	\node at (11,-1) {$a_3$};
	\node at (11,-1.5) {$b_3$};
	\node at (12,-1) {$a_4$};
	\node at (13.5,-1) {$\psi_4$};
	\node at (15,-1) {$b_4$};
	\node at (16,-1) {$\phi_4$};
	\draw (4) .. controls (4.25,.75) and (4.75,.75) .. (5);
	\draw (7) .. controls (6.5, 1.25) and (7.5,1.25) .. (7);
	\draw (11) .. controls (10.5, 1.25) and (11.5,1.25) .. (11);
	\draw (12) .. controls (12.75,1.75) and (14.25,1.75) .. (15);
	\draw [|-|] (.7,-.5) -- (3.3,-.5); 
	\draw [|-|] (7.7,-.5) -- (10.3,-.5); 
	\draw [|-|] (12.7,-.5) -- (14.3,-.5); 
 \end{tikzpicture}.
 $$
 
 Then
 $$b_0\phi_0a_1=(17,11,12,10,7),\ b_1\phi_1a_2=(8,9,1),\ b_2\phi_2a_2=(1,6,13,14,16),$$
  $$b_3\phi_3a_3=(16,2),\  b_4\phi_4a_4=(5,15,0),$$
so $s=4$,
$$
\mathrm{mt}_\tau^\mu(\phi)=0+0+3+0+0=3,\quad  \mathrm{io}_\tau^\mu(\phi)=1+0+0+0+0=1,\quad \text{and}
$$
$$
 \mathrm{rst}_\tau^\mu(\phi) = 3+1+3+0+1-3-1=4
$$
Thus,
$$z_{\phi,\tau}^\mu(t)=(t-1)^3t^4(t+1),$$
and
\begin{align*}
\mathrm{Coeff}(S(\chi^{(\UIO_{16},1\slarc{}16)}),\chi^{(\phi,\mu)}) &= (-1)^4t^3(t-1)^3t^4(t+1)\\
&=(t-1)^3t^7(t+1).
\end{align*}

A nice consequence is the case where $\mu=\emptyset$.

\begin{corollary} \label{CorollaryTrivialCoefficient} For a finite set $K$ and $\phi,\tau\in L[K]$
$$\mathrm{Coeff}(S(\chi^{(\tau,{1_\tau}\!\!\raise2pt\hbox{$\slarc{}$}m_\tau)}),\chi^{(\phi,\emptyset)})=(q-1)z_{\phi,\tau}^\emptyset(q-1).$$
\end{corollary}

\begin{remark} Under the functor $\overline{\mathcal K}$, Corollary~\ref{CorollaryTrivialCoefficient} gives
  $$\mathrm{Coeff}(S(\chi^{1\slarc{}n}),\chi^\emptyset)=\sum_{\phi\in L[n]}(q-1)z_{\phi,\UIO_n}(q-1),$$
and there is no cancelation in this sum. However, if we expand the right hand side as a polynomial in $t=q-1$ and multiply by $(-1)^n$, we discover that all the coefficients of $t^k$ are positive for all $k$. This can be shown using a simple involution on the $\phi$'s but we do not have a good combinatorial interpretation for the coefficient of $t^k$ in that expansion.
\end{remark}

 
\section{Primitives}\label{sec:primitive}

In this section we describe the Lie monoid of primitives of ${{\bf scf}(U)}$. We refer the reader to \cite{AM} and the references therein for more details on the concepts of Lie monoid, primitives and the Cartier--Milnor--Moore theorem for Hopf monoid.
In~\cite{ABT} we already give a description of the primitive space using the Eulerian idempotent, but our description there is not explicit.
Here we give an explicit description using an analogue of the Dynkin idempotents.
To achieve our goal  we construct an explicit basis
of the Lie monoid. It is best to start from the basis $P^{\le}_{(\phi,\lambda)}$ as described in Section~\ref{sec:Pbasis}. 
To simplify the notation we will denote this
basis with $P$ and suppress  a fixed power-sum friendly order $\le$ from the notation.
In the Hopf monoid ${{\bf scf}(U)}$ there is a natural Lie monoid  structure given by the Lie bracket
 $$[\,,\,]_{I,J}= m_{I,J}-m_{J,I}\circ\beta_{I,J}\colon {{\bf scf}(U)}[I]\otimes{{\bf scf}(U)}[J]\to{{\bf scf}(U)}[K]$$
 where $K=I\sqcup J$. For any finite set $K$, the subspace of primitive ${\mathcal P}({{\bf scf}(U)}[K])$ of ${{\bf scf}(U)}[K]$ is the set
\begin{equation}\label{eq:defprim}
  {\mathcal P}({{\bf scf}(U)}[K])=\big\{Q\in{{\bf scf}(U)}[K]\mid \Delta_{I,J}(Q)=0; \, (I,J)\models K\big\}.
\end{equation}   
It is straightforward to check that ${\mathcal P}({{\bf scf}(U)}[K])$ is a sub-Lie monoid (closed under the Lie bracket above).

Since ${{\bf scf}(U)}$ is a cocomutative Hopf monoid, the Cartier--Milnor--Moore Theorem gives us a functor $\mathcal U$ (analoguous to the universal enveloping algebra for Lie algebra) such that
  $${\mathcal U}\big( {\mathcal P}({{\bf scf}(U)})\big)= {{\bf scf}(U)}.$$
  In particular, to find a basis for ${\mathcal P}({{\bf scf}(U)})$ it is sufficient to find algebraically independent elements in each ${\mathcal P}({{\bf scf}(U)}[K])$ that generates $ {{\bf scf}(U)}$.
  Then, the Cartier--Milnor--Moore Theorem gives that ${\mathcal P}({{\bf scf}(U)}[K])$ is the free Lie monoid generated by the algebraically independent primitives. A basis of ${\mathcal P}({{\bf scf}(U)}[K])$ is then the Lyndon bracket of the algebraically independent primitives.

Our first step, given any cocomutative Hopf monoid $\bf H$, is to construct  projection operators  $\Psi\colon {\bf H}\to {\mathcal P}({\bf H})$. This will be useful to obtain algebraically independent primitives. 

 For any  $k\in K$, (as in and Forest~\cite{F}), we  define
  \begin{equation}\label{eq:PrimProj}
    \Psi_{(k,K)}  = \sum_{J=(J_1,\ldots, J_\ell)\models K \atop k\in J_1} (-1)^{\ell-1}  m_J\circ \Delta_J
 \end{equation}
 
\begin{theorem}\label{thm:PrimProj} For any $k\in K$, we have a projection 
 $$ \Psi_{(k,K)}\colon {\bf H}[K]\to {\mathcal P}({\bf H})[K].$$
\end{theorem}

\begin{proof} To have that $\Psi_{(k,K)}$ maps into ${\mathcal P}({\bf H})[K]$, we need to show that $\Delta_{A,B}\circ \Psi_{(k,K)}=0$ for all $(A, B)\models K$ (hence, $A,B\ne\emptyset$).
For $J=(J_1, J_2, \ldots, J_\ell)\models K$, several applications of the compatibility condition between $\Delta$ and $m$, followed by coassociativity and cocommutativity show that
  $$\Delta_{A,B}\circ m_J\circ \Delta_J =\big((m_{J\cap A}\circ \Delta_{J\cap A})\otimes (m_{J\cap B}\circ \Delta_{J\cap B})\big)\circ \Delta_{A,B}\,,
  $$
where $J\cap A=(J_1\cap A, \ldots,J_\ell\cap A)$ and same for $J\cap B$. We thus have
$$
   \Delta_{A,B}\circ \Psi_{(k,K)}  = \sum_{J=(J_1,\ldots, J_\ell)\models K \atop k\in J_1} (-1)^{\ell-1}  \big((m_{J\cap A}\circ \Delta_{J\cap A})\otimes (m_{J\cap B}\circ \Delta_{J\cap B})\big)\circ \Delta_{A,B}.
$$
We now describe an involution that shows that all the terms on the right hand side cancel. Let us first assume that $k\in A$. 
Given $J=(J_1, J_2, \ldots, J_\ell)\models K$ with $k\in J_1$, let $\beta=\min\{i: J_i\cap B\ne \emptyset\}$. Consider the map
  $$  \omega(J) = \left\{  \begin{array}{ll}
     (J_1,\ldots,  J_{\beta-1}\cup J_\beta  ,\ldots, J_\ell)  &\hbox{if $J_\beta\cap B=J_\beta$,}\\ &\\
     (J_1,\ldots,  J_\beta \cap A, J_\beta\cap B, \ldots, J_\ell)  &\hbox{if $J_\beta\cap B\ne J_\beta$.} \end{array}
           \right. $$
Note that this is a well defined involution, since if $k\in J_1\cap A$, if $J_\beta\cap B=J_\beta$, then we must have that $\beta>1$ and $J_{\beta-1}\cap A=J_{\beta-1}$.
If $J_\beta\cap B\ne J_\beta$, then we create two non-empty part $J_\beta \cap A\mid J_\beta\cap B$ and $J_\beta\cap B$ is still the first part of the resulting set composition with non-empty intersection with $B$. Thus,
$$(m_{J\cap A}\circ \Delta_{J\cap A})\otimes (m_{J\cap B}\circ \Delta_{J\cap B}) = 
    (m_{\omega(J)\cap A}\circ \Delta_{\omega(J)\cap A})\otimes (m_{\omega(J)\cap B}\circ \Delta_{\omega(J)\cap B})$$
and the terms have opposite signs. The case where $k\in B$ is similar, interchanging the role of $A$ and $B$.
This shows we get zero and that  $\Psi_{(k,K)}$ maps into ${\mathcal P}({\bf H})[K]$.

Since each $y=\Psi_{(k,K)}(x)\in {\mathcal P}({\bf H})[K]$ is primitive, $\Delta_J(y)=0$ unless $J=(K)$, so 
 $$\Psi_{(k,K)}(y)=\sum_{J=(J_1,\ldots, J_\ell)\models K \atop k\in J_1} (-1)^{\ell-1}  m_J\circ \Delta_J(y)=y.$$
We conclude that $\Psi_{(k,K)}$ is a projection.
\end{proof}

 We can now construct a primitives element $Q_{(\phi,\lambda)}$ for each $(\phi,\lambda)$ atomic. The fact that these primitives are algebraically independent and span $ {{\bf scf}(U)}$ will follow from triangularity relation and Corollary~\ref{cor:Patomic}. Let $1_\phi$ be as in (\ref{MinMaxElements}), and
 \begin{equation}\label{eq:primQ}
    Q_{(\phi,\lambda)}  = \Psi_{(1_\phi,K)}(P_{(\phi,\lambda)}).
 \end{equation}

This expression can be simplified in our case.
Given $\phi\in L[K]$ and $\lambda\in\cS^\phi$, let $\lambda=\mu_1\cup\mu_2\cup\cdots\cup\mu_\ell$ be the maximal decomposition of $\lambda$ into connected components. That is, there exists a set partition $I=\{I_1, I_2,\ldots,I_\ell\}$ of $K$  such that for each $1\leq i\leq \ell$,
$$ \mu_i=\{j^i_0\larc{}j^i_1,j^i_1\larc{}j^i_2,\ldots, j^i_{d_i-1}\larc{}j^i_{d_i}\} \quad\text{where}\quad\phi|_{I_i}=({j^i_0},{j^i_1},\ldots,{j^i_{d_i}})
.$$
Note that if $d_i=0$, then $\mu_i=\emptyset$ is a single dot. From the definition of the $P$ basis, we have that $\Delta_J(P_{(\phi,\lambda)})=0$ unless 
    $$ I_i \cap J_j \ne \emptyset \quad\text{implies}\quad I_i\cap J_j = I_i.$$
So $\Delta_J(P_{(\phi,\lambda)})\ne 0$ if and only if we can find $A=(A_1,\ldots, A_k)\models [\ell]$ such that $J=(J_{A_1},\ldots, J_{A_k})\models K$, where $J_{A_j}=\bigcup_{i\in A_j} I_i$. For convenience we  may assume that $1_\phi\in I_1$, so
 \begin{equation}\label{eq:primQ2}
    Q_{(\phi,\lambda)}  = \sum_{A=(A_1,\ldots, A_k)\models [\ell] \atop 1_\phi\in A_1} (-1)^{k-1}  (m_{(J_{A_1},\ldots, J_{A_k})}\circ \Delta_{(J_{A_1},\ldots, J_{A_k})})(P_{(\phi,\lambda)}).
 \end{equation}
Note that the basis $P$ is such that $(m_{(J_{A_1},J_{A_2},\ldots,J_{A_k})}\circ \Delta_{(J_{A_1},J_{A_2},\ldots, J_{A_k})})(P_{(\phi,\lambda)})$ is a single basis element $P_{(\phi',\lambda)}$ obtained from $(\phi,\lambda)$ by reorganizing the components $\mu_i$'s according to $A$.
For example, let 
 $$(\phi,\lambda)=
  \begin{tikzpicture}[baseline=.2cm]
	\foreach \x in {1,2,3,4,5,6,7} 
		\node (\x) at (\x/4,0) [inner sep=-1pt] {$\scriptstyle \bullet$};
	\node at (1/4,-.2) {$\scriptscriptstyle 2$};
	\node at (2/4,-.2) {$\scriptscriptstyle 4$};
	\node at (3/4,-.2) {$\scriptscriptstyle 5$};
	\node at (4/4,-.2) {$\scriptscriptstyle 6$};
	\node at (5/4,-.2) {$\scriptscriptstyle 1$};
	\node at (6/4,-.2) {$\scriptscriptstyle 3$};
	\node at (7/4,-.2) {$\scriptscriptstyle 7$};
	\draw  (1) .. controls (1.5/4,.75/2) and (2.5/4,.75/2) ..  node [above=-2pt] { } (3); 
	\draw  (2) .. controls (2.5/4,.75/2) and (3.5/4,.75/2) ..  node [above=-2pt] { } (4); 
	\draw  (3) .. controls (4/4,1/2) and (6/4,1/2) ..  node [above=-2pt] { } (7); 
	\draw  (4) .. controls (4.25/4,.5/2) and (4.75/4,.5/2) ..  node [above=-2pt] { } (5); 
\end{tikzpicture}$$
Then, $I=(\{2,5,7\},\{4,6,1\},\{3\})\models [6]$ and we have 
  $$
 (\phi|_{I_1},\mu_1)=
  \begin{tikzpicture}[baseline=0.0cm]
	\foreach \x in {1,2,3} 
		\node (\x) at (\x/4,0) [inner sep=-1pt] {$\scriptstyle \bullet$};
	\node at (1/4,-.2) {$\scriptscriptstyle 2$};
	\node at (2/4,-.2) {$\scriptscriptstyle 5$};
	\node at (3/4,-.2) {$\scriptscriptstyle 7$};
	\draw  (1) .. controls (1.25/4,.5/2) and (1.75/4,.5/2) ..  node [above=-2pt] { } (2); 
	\draw  (2) .. controls (2.25/4,.5/2) and (2.75/4,.5/2) ..  node [above=-2pt] { } (3); 
\end{tikzpicture},\quad
(\phi|_{I_2},\mu_2)=
  \begin{tikzpicture}[baseline=0.0cm]
	\foreach \x in {1,2,3} 
		\node (\x) at (\x/4,0) [inner sep=-1pt] {$\scriptstyle \bullet$};
	\node at (1/4,-.2) {$\scriptscriptstyle 4$};
	\node at (2/4,-.2) {$\scriptscriptstyle 6$};
	\node at (3/4,-.2) {$\scriptscriptstyle 1$};
	\draw  (1) .. controls (1.25/4,.5/2) and (1.75/4,.5/2) ..  node [above=-2pt] { } (2); 
	\draw  (2) .. controls (2.25/4,.5/2) and (2.75/4,.5/2) ..  node [above=-2pt] { } (3); 
\end{tikzpicture},\quad
 (\phi|_{I_3},\mu_3)=
  \begin{tikzpicture}[baseline=0.0cm]
	\foreach \x in {1} 
		\node (\x) at (\x/4,0) [inner sep=-1pt] {$\scriptstyle \bullet$};
	\node at (1/4,-.2) {$\scriptscriptstyle 3$};
\end{tikzpicture}.
$$
Here, $\ell=3$ and the set composition $A$ of $\{1,2,3\}$ with 1 in the first part are
$$ (\{123\});\ (\{1,2\},\{3\});\ (\{1,3\},\{2\});\ (\{1\},\{2,3\}); \ (\{1\}, \{2\}, \{3\});\ (\{1\}, \{3\}, \{2\}).
$$
Each gives us a term of $Q_{(\lambda,\phi)}$ as follows
\begin{align*}
   Q_{(\lambda,\phi)}= \ \,&
   P_{ \begin{tikzpicture}[baseline=-0.2cm]
	\foreach \x in {1,2,3,4,5,6,7} 
		\node (\x) at (\x/4,0) [inner sep=-1pt] {$\scriptstyle \bullet$};
	\node at (1/4,-.2) {$\scriptscriptstyle 2$};
	\node at (2/4,-.2) {$\scriptscriptstyle 4$};
	\node at (3/4,-.2) {$\scriptscriptstyle 5$};
	\node at (4/4,-.2) {$\scriptscriptstyle 6$};
	\node at (5/4,-.2) {$\scriptscriptstyle 1$};
	\node at (6/4,-.2) {$\scriptscriptstyle 3$};
	\node at (7/4,-.2) {$\scriptscriptstyle 7$};
	\draw  (1) .. controls (1.5/4,.75/2) and (2.5/4,.75/2) ..  node [above=-2pt] { } (3); 
	\draw  (2) .. controls (2.5/4,.75/2) and (3.5/4,.75/2) ..  node [above=-2pt] { } (4); 
	\draw  (3) .. controls (4/4,1/2) and (6/4,1/2) ..  node [above=-2pt] { } (7); 
	\draw  (4) .. controls (4.25/4,.5/2) and (4.75/4,.5/2) ..  node [above=-2pt] { } (5); 
   \end{tikzpicture}}  \!\! -\
   P_{ \begin{tikzpicture}[baseline=-0.2cm]
	\foreach \x in {1,2,3,4,5,6,7} 
		\node (\x) at (\x/4,0) [inner sep=-1pt] {$\scriptstyle \bullet$};
	\node at (1/4,-.2) {$\scriptscriptstyle 2$};
	\node at (2/4,-.2) {$\scriptscriptstyle 4$};
	\node at (3/4,-.2) {$\scriptscriptstyle 5$};
	\node at (4/4,-.2) {$\scriptscriptstyle 6$};
	\node at (5/4,-.2) {$\scriptscriptstyle 1$};
	\node at (6/4,-.2) {$\scriptscriptstyle 7$};
	\node at (7/4,-.2) {$\scriptscriptstyle 3$};
	\draw  (1) .. controls (1.5/4,.75/2) and (2.5/4,.75/2) ..  node [above=-2pt] { } (3); 
	\draw  (2) .. controls (2.5/4,.75/2) and (3.5/4,.75/2) ..  node [above=-2pt] { } (4); 
	\draw  (3) .. controls (3.75/4,1/2) and (5.25/4,1/2) ..  node [above=-2pt] { } (6); 
	\draw  (4) .. controls (4.25/4,.5/2) and (4.75/4,.5/2) ..  node [above=-2pt] { } (5); 
   \end{tikzpicture}} 
 \!\! -\
   P_{ \begin{tikzpicture}[baseline=-0.2cm]
	\foreach \x in {1,2,3,4,5,6,7} 
		\node (\x) at (\x/4,0) [inner sep=-1pt] {$\scriptstyle \bullet$};
	\node at (1/4,-.2) {$\scriptscriptstyle 2$};
	\node at (2/4,-.2) {$\scriptscriptstyle 5$};
	\node at (3/4,-.2) {$\scriptscriptstyle 3$};
	\node at (4/4,-.2) {$\scriptscriptstyle 7$};
	\node at (5/4,-.2) {$\scriptscriptstyle 4$};
	\node at (6/4,-.2) {$\scriptscriptstyle 6$};
	\node at (7/4,-.2) {$\scriptscriptstyle 1$};
	\draw  (1) .. controls (1.25/4,.5/2) and (1.75/4,.5/2) ..  node [above=-2pt] { } (2); 
	\draw  (2) .. controls (2.5/4,.75/2) and (3.5/4,.75/2) ..  node [above=-2pt] { } (4); 
	\draw  (5) .. controls (5.25/4,.5/2) and (5.75/4,.5/2) ..  node [above=-2pt] { } (6); 
	\draw  (6) .. controls (6.25/4,.5/2) and (6.75/4,.5/2) ..  node [above=-2pt] { } (7); 
   \end{tikzpicture}}   \!\! -\
   P_{ \begin{tikzpicture}[baseline=0.0cm]
	\foreach \x in {1,2,3,4,5,6,7} 
		\node (\x) at (\x/4,0) [inner sep=-1pt] {$\scriptstyle \bullet$};
	\node at (1/4,-.2) {$\scriptscriptstyle 2$};
	\node at (2/4,-.2) {$\scriptscriptstyle 5$};
	\node at (3/4,-.2) {$\scriptscriptstyle 7$};
	\node at (4/4,-.2) {$\scriptscriptstyle 4$};
	\node at (5/4,-.2) {$\scriptscriptstyle 6$};
	\node at (6/4,-.2) {$\scriptscriptstyle 1$};
	\node at (7/4,-.2) {$\scriptscriptstyle 3$};
	\draw  (1) .. controls (1.25/4,.5/2) and (1.75/4,.5/2) ..  node [above=-2pt] { } (2); 
	\draw  (2) .. controls (2.25/4,.5/2) and (2.75/4,.5/2) ..  node [above=-2pt] { } (3); 
	\draw  (4) .. controls (4.25/4,.5/2) and (4.75/4,.5/2) ..  node [above=-2pt] { } (5); 
	\draw  (5) .. controls (5.25/4,.5/2) and (5.75/4,.5/2) ..  node [above=-2pt] { } (6); 
      \end{tikzpicture}} 
\\
   &+\  
   P_{ \begin{tikzpicture}[baseline=0.0cm]
	\foreach \x in {1,2,3,4,5,6,7} 
		\node (\x) at (\x/4,0) [inner sep=-1pt] {$\scriptstyle \bullet$};
	\node at (1/4,-.2) {$\scriptscriptstyle 2$};
	\node at (2/4,-.2) {$\scriptscriptstyle 5$};
	\node at (3/4,-.2) {$\scriptscriptstyle 7$};
	\node at (4/4,-.2) {$\scriptscriptstyle 4$};
	\node at (5/4,-.2) {$\scriptscriptstyle 6$};
	\node at (6/4,-.2) {$\scriptscriptstyle 1$};
	\node at (7/4,-.2) {$\scriptscriptstyle 3$};
	\draw  (1) .. controls (1.25/4,.5/2) and (1.75/4,.5/2) ..  node [above=-2pt] { } (2); 
	\draw  (2) .. controls (2.25/4,.5/2) and (2.75/4,.5/2) ..  node [above=-2pt] { } (3); 
	\draw  (4) .. controls (4.25/4,.5/2) and (4.75/4,.5/2) ..  node [above=-2pt] { } (5); 
	\draw  (5) .. controls (5.25/4,.5/2) and (5.75/4,.5/2) ..  node [above=-2pt] { } (6); 
      \end{tikzpicture}} 
 \!\! +\
        P_{ \begin{tikzpicture}[baseline=0.0cm]
	\foreach \x in {1,2,3,4,5,6,7} 
		\node (\x) at (\x/4,0) [inner sep=-1pt] {$\scriptstyle \bullet$};
	\node at (1/4,-.2) {$\scriptscriptstyle 2$};
	\node at (2/4,-.2) {$\scriptscriptstyle 5$};
	\node at (3/4,-.2) {$\scriptscriptstyle 7$};
	\node at (4/4,-.2) {$\scriptscriptstyle 3$};
	\node at (5/4,-.2) {$\scriptscriptstyle 4$};
	\node at (6/4,-.2) {$\scriptscriptstyle 6$};
	\node at (7/4,-.2) {$\scriptscriptstyle 1$};
	\draw  (1) .. controls (1.25/4,.5/2) and (1.75/4,.5/2) ..  node [above=-2pt] { } (2); 
	\draw  (2) .. controls (2.25/4,.5/2) and (2.75/4,.5/2) ..  node [above=-2pt] { } (3); 
	\draw  (5) .. controls (5.25/4,.5/2) and (5.75/4,.5/2) ..  node [above=-2pt] { } (6); 
	\draw  (6) .. controls (6.25/4,.5/2) and (6.75/4,.5/2) ..  node [above=-2pt] { } (7); 
      \end{tikzpicture}} \,.
\end{align*}
Note that the terms corresponding to $(\{1\}, \{2\}, \{3\})$ and $(\{1\}, \{2,3\})$ are the same and cancel with each other. This is due to the fact that $(\phi|_{I_2\cup I_3},\mu_2\cup\mu_3)$ is not atomic. 

Since the process of regrouping the connected components $\mu_i$ only reduces the size of the arcs in $\mu_i$,  the element $P_{(\phi',\lambda)}=(m_{J_{A_1},J_{A_2},\ldots,J_{A_k}}\circ \Delta_{J_{A_1},J_{A_2},\ldots, J_{A_k}})(P_{(\phi,\lambda)})$ satisfies $\dim(\phi,\lambda)\ge \dim(\phi',\lambda)$ for any $A\models[\ell]$.  If $(\phi,\lambda)$ is atomic and $A\ne([\ell])$, then $(\phi',\lambda)$ is not atomic so $\dim(\phi,\lambda)> \dim(\phi',\lambda)$. Thus, 
when $(\phi,\lambda)$ is atomic, we obtain a triangularity relation
$$ Q_{(\phi,\lambda)}  = P_{(\phi,\lambda)} + \sum_{\dim(\phi',\lambda)>\dim(\phi,\lambda)} c_{\phi,\phi'} P_{(\phi',\lambda)}
$$
In particular, in this case $Q_{(\phi,\lambda)}\ne 0$. We have

\begin{theorem}\label{thm:PrimQ3} The $Q_{(\phi,\lambda)}$ for $(\phi,\lambda)$ atomic form a full system of algebraically independent primitive elements that span ${\bf scf}(U)$.
\end{theorem}
\begin{proof}
 For any $(\phi,\lambda)$, let $(\phi,\lambda)=(\phi_1,\lambda_1)\cdots(\phi_r,\lambda_r)$ be its unique factorization into atomics. We have that
\begin{align*}
     Q_{(\phi_1,\lambda_1)} \cdots Q_{(\phi_r,\lambda_r)} 
                        &= P_{(\phi,\lambda)} + \sum_{\dim(\phi'_i,\lambda_i)\ge\dim(\phi_i,\lambda_i)} c_{\phi',\phi} P_{(\phi',\lambda)}, 
\end{align*}
where in the sum at least one inequality is strict. This shows that ${\bf scf}(U)$ is freely generated by the $Q_{(\phi,\lambda)}$ for $(\phi,\lambda)$ atomic.
\end{proof}
  
  \begin{remark} Our Theorem~\ref{thm:PrimQ3} is closely related to the analogue theorems of Lauve--Masnak~\cite{LM}. 
  It is not cancelation free; cancelation occurs when a non-atomic is nested under an arc.
  \end{remark}

  \begin{remark} Using the functor $\overline{\mathcal K}$, we recover the formula of~\cite{LM} in $\overline{\mathcal K}({{\bf scf}(U)})$ from (\ref{eq:primQ}). The formula in $\overline{\mathcal K}({{\bf scf}(U)})$ will generally have multiplicities but no further cancelations. 
    \end{remark}

\section*{Acknowledgements} 

We would like to thanks E. Kruskal for help proving Lemma \ref{(n)Lemma} and G. Grell for contributing to the proof of Theorem \ref{PRestrictionTheorem}.


\end{document}